\documentclass[11pt]{article}
\usepackage[margin=0.850in]{geometry}

\usepackage{amsmath} % assumes amsmath package installed
\usepackage{amssymb}  % assumes amsmath package installed
\usepackage{amsthm}
\usepackage{hyperref} % Amirali: I commented this package out because after 22 equations, it wasn't
\usepackage{color}
\usepackage{graphicx}        % standard LaTeX graphics tool
\usepackage[hang,tight]{subfigure}

\usepackage{multicol}

\usepackage{algorithm}

\theoremstyle{plain}  % Bold name, italics font
\newtheorem{theorem}{Theorem}[section]

\newtheorem{definition}[theorem]{Definition}

\theoremstyle{definition}

\theoremstyle{remark} % italics name, roman font

\newtheorem{remark}{Remark}[section]

\title{\LARGE \bf Some Applications of Polynomial Optimization \\
in Operations Research and Real-Time Decision Making
}
 \author{\Large Amir Ali Ahmadi\thanks{Amir Ali Ahmadi is with the Department of Operations Research and Financial Engineering at Princeton University. He has been partially supported for this work by the AFOSR Young Investigator Program Award.\newline \href{mailto:aaa@us.ibm.com}{a\_a\_a@princeton.edu}, \url{http://aaa.princeton.edu/}} and Anirudha Majumdar\thanks{Anirudha Majumdar is with the Department of Electrical Engineering and Computer Science, CSAIL, MIT.\newline   \href{mailto:anirudha@mit.edu}{anirudha@mit.edu}, \url{http://www.mit.edu/\~anirudha}} }
\begin{document}
\date{}
\maketitle

%\thispagestyle{empty}
%\pagestyle{empty}

%%%%%%%%%%%%%%%%%%%%%%%%%%%%%%%%%%%%%%%%%%%%%%%%%%%%%%%%%%%%%%%%%%%%%%%%%%%%%%%%
\begin{abstract}
We demonstrate applications of algebraic techniques that optimize and certify polynomial inequalities to problems of interest in the operations research and transportation engineering communities. Three problems are considered: (i) wireless coverage of targeted geographical regions with guaranteed signal quality and minimum transmission power, (ii) computing real-time certificates of collision avoidance for a simple model of an unmanned vehicle (UV) navigating through a cluttered environment, and (iii) designing a nonlinear hovering controller for a quadrotor UV, which has recently been used for load transportation. On our smaller-scale applications, we apply the sum of squares (SOS) relaxation and solve the underlying problems with semidefinite programming. On the larger-scale or real-time applications, we use our recently introduced ``SDSOS Optimization'' techniques which result in second order cone programs. To the best of our knowledge, this is the first study of real-time applications of sum of squares techniques in optimization and control. No knowledge in dynamics and control is assumed from the reader.
%There are several applications in operations research and in particular in military operations that can be expressed as optimization problems over the set of nonnegative polynomials. It is fair to say, however, that research on these applications has so far remained closer to theory than to practice. The reason for this is perhaps twofold: (i) optimizing over nonnegative polynomials is a fundamentally difficult problem (it is NP-hard even for degree four polynomials); more importantly, (ii) the popular approximation algorithms based on sum of squares (SOS) and semidefinite programming often struggle to scale to problem sizes of interest in applications. In this paper, we show the promise of \emph{``DSOS and SDSOS Optimization''} in moving closer to larger-scale military applications of polynomial optimization. These are alternatives to sum of squares optimization developed by the authors that are based on linear programming and second programming. Significant improvements in scalability are demonstrated on three problems of interest in combat operations: (i) optimal jamming of a wireless network, (ii) safety verification of an F/A-18 Hornet aircraft, and (iii) ****.
\end{abstract}

%\keywords{Sum of squares optimization, polynomial optimization, nonnegative polynomials, optimization over symmetric matrices, semidefinite programming, linear programming, second order cone programming.}

%%%%%%%%%%%%%%%%%%%%%%%%%%%%%%%%%%%%%%%%%%%%%%%%%%%%%%%%%%%%%%%%%%%%%%%%%%%%%%%%

\section{Introduction}
In this paper we consider applications of \emph{polynomial optimization} in the area of operations research and transportation engineering. While techniques in more established areas of optimization theory such as linear, integer, combinatorial, and dynamic programming have found wide applications in these areas \cite{bertsimas1997introduction,bertsekas1995dynamic,berman1995location,Karlof05}, the relatively newer field of polynomial optimization, which has gone through rapid advancements in recent years, may yet prove to reveal many unexplored applications. It is our aim in this paper to bring a few such applications to the attention of the operation research community and to highlight some algorithmic tools based on algebraic techniques that we believe are particularly suited for approaching problems of this sort.

%has touched fewer borders with applications in operations research. We believe that a clear exposition of the algebraic techniques from the theory of polynomial optimization could lead to an exciting set of new applications. To this end, in this paper we highlight some important theoretical and algorithmic ideas from the area that we believe have particular promise and present a few examples that are of potential interest to practitioners in operations research and transportation engineering. 

The fundamental problem underlying all of our applications is that of \emph{optimizing over nonnegative polynomials}. This is the task of finding the coefficients $c_\alpha\mathrel{\mathop:}=c_{\alpha_1,\ldots,\alpha_n}$ of some multivariate polynomial $p(x)\mathrel{\mathop:}=p(x_1,\ldots,x_n)=\sum_{\alpha} c_\alpha x^\alpha$ in order to get $p(x)\geq 0$, either globally (i.e., $\forall x\in\mathbb{R}^n$), or on certain basic semialgebraic sets. A \emph{basic semialgebraic set} is a subset of the Euclidean space defined by a finite number of polynomial (in)equalities. That is a set of the form $$\mathcal{S}\mathrel{\mathop:}=\{x\in\mathbb{R}^n|\  g_i(x)\geq 0, h_i(x)=0\},$$ where the functions $g_i,h_i$ are all multivariate polynomials. The polynomial optimization problem (POP) is itself a problem of this form. Indeed, the task of finding the minimum of a polynomial function $q$ on a basic semialgebraic set $\mathcal{S}$ is the same as that of finding the largest constant $\gamma$ such that $q(x)-\gamma$ is nonnegative on $\mathcal{S}$. There are, however, many other applications of optimization over nonnegative polynomials, some to be seen in this work.

Our paper is organized as follows. In Section \ref{sec:math.background}, we briefly review the concept of sum of squares (sos) decomposition and its relation to semidefinite programming (SDP). This is a popular approach for certifying polynomial nonnegativity. While remarkably powerful, it often faces scalability limitations on larger-scale problems. As a potential remedy, we have recently introduced \cite{Ahmadi14,dsos_ciss14} the concepts of \emph{diagonally dominant and scaled diagonally dominant sum of squares (dsos and sdsos)} decomposition, which instead of SDP result in linear programs (LP) and second order cone programs (SOCP) respectively. These concepts are also presented in Section \ref{sec:math.background}.

In Section \ref{sec:jamming}, we consider the problem of providing guaranteed wireless coverage to certain basic semialgebraic subsets of the Euclidean space with minimum transmission power. 
%
%optimal robust coverage in a wireless network where the transmitters have limits on their energy usage. 
The general problem here has been previously considered in the literature but we show that tools from polynomial optimization allow us to handle the problem in broader and arguably more realistic scenarios.
%
%is a problem area which has previously been considered in the literature, but we show the benefits that tools from polynomial optimization can bring to the problem. 
Our next two examples are related to transportation problems in operations research. In particular, in Section \ref{sec:barriers}, we consider a simple model of an unmanned aerial vehicle (UAV), which aims to fly through a cluttered environment in a collision free manner. The techniques presented in this section can also be adapted for applications to ground vehicles. We demonstrate how one can choose a control law and at the same time find a \emph{formal certificate}---an independently verifiable proof---that the resulting dynamics will guarantee no collisions with obstacles. We show that using our SOCP techniques, the underlying computational task can be carried out in the order of 20-30 milliseconds, hence making a plausibility claim about a real-time application of this approach. In Section \ref{sec:quadrotor}, we use the same technical tools to design a stabilizing controller for a quadrotor system, a device that has increasing potential for use in transportation (see Section~\ref{sec:quadrotor}). The designed controller prevents the quadrotor from losing balance when it is subject to environmental disturbance, or an external perturbation. The SDP resulting from this controller design problem is so large that it cannot be solved on our machine (3.4 GHz PC with 4 cores and 16 GB RAM). This is another example demonstrating the promise of our new sdsos machinery for handling problems of large scale.

Our second and third applications include the employment of Lyapunov techniques to convert a problem in dynamics and control to a problem in polynomial optimization. Since we do not want to assume this background from the reader, we present the essentials of these very basic concepts in Section \ref{sec:Lyap.background}. The mathematical background in this section (just like Section~\ref{sec:math.background}) is presented at a minimal level to make the paper self-contained, while keeping the focus on the applications and the algorithmic aspects. We end the paper with some brief concluding remarks in Section~\ref{sec:conclusions}.

%In general, we have made an attempt to present the mathematical background (Sections~\ref{sec:math.background} and~\ref{sec:Lyap.background}) at the minimum level required to make the paper self-contained, while keeping the focus on the applications and the algorithmic aspects.

%As a result, almost all proofs in these two sections are omitted and can be found in the references. We end the paper with some brief concluding remarks in Section~\ref{sec:conclusions}.

% Finally, some concluding remarks are presented in Section \ref{sec:conclusions}.

\section{Algebraic certificates of nonnegativity via convex optimization}\label{sec:math.background}

The task of optimizing over nonnegative polynomials or even checking nonnegativity of a given polynomial, either globally or on a basic semialgebraic set, is known to be NP-hard~\cite{nonnegativity_NP_hard}. This is true already for checking global nonnegativity of a quartic (degree-4) polynomial, or for checking nonnegativity of a quadratic polynomial on a set defined by linear inequalities. A popular relaxation scheme for this problem is through the machinery of the so-called \emph{sum of squares optimization}.

We say that a polynomial $p$ is a sum of squares (sos), if it can be written as $p=\sum_iq^2$ for some other polynomials $q_i$. Obviously, such a decomposition is a sufficient (but in general not necessary~\cite{Hilbert_1888}) condition for (global) nonnegativity of $p$. 
%We know since as far back as the work of Hilbert~\cite{Hilbert_1888} in 1888 that this condition is not necessary for nonnegativity; however, 
%not every nonnegative polynomial is a sum of squares. The modern advantage, however, is that we can check for existence of a sum of squares decomposition via semidefinite programming.
The situation where $p$ is only constrained to be nonnegative on a certain basic semialgebraic set\footnote{In this formulation, we have avoided equality constraints for simplicity. Obviously, there is no loss of generality in doing this as an equality constraint $h(x)=0$ can be imposed by the pair of inequality constraints $h(x)\geq 0, -h(x)\geq 0$.} $$\mathcal{S}\mathrel{\mathop:}=\{x\in\mathbb{R}^n| \ g_1(x)\geq 0,\ldots,g_m(x)\geq 0\}$$ can also be handled with the help of appropriate sum of squares multipliers. For example, if we succeed in finding sos polynomials $s_0,s_1,\ldots,s_m$, such that

\begin{equation} \label{eq:p=s0+sigi}
p(x)=s_0(x)+\sum_{i=1}^m s_i(x)g_i(x),
\end{equation}
then we have found a certificate of nonnegativity of $p$ on the set $\mathcal{S}$. Indeed, if we evaluate the above expression at any $x\in\mathcal{S}$, nonnegativity of the polynomials $s_0,s_1\ldots,s_m$ imply that $p(x)\geq 0$. A
Positivstellensatz theorem from real algebraic geometry due to Putinar~\cite{putinar1993positive} states that if the set $\mathcal{S}$ satisfies the so-called \emph{Archimedean} property, a property only slightly stronger than compactness\footnote{In particular, if we have as an outer estimate a ball of some radius $R$ in which our set $\mathcal{S}$ lives, then we can add a single quadratic inequality $\sum_i x_i^2\leq R$ to the description of $\mathcal{S}$ to have it satisfy the Archimedean property without changing the set.}, then every polynomial positive on $\mathcal{S}$ has a representation of the type (\ref{eq:p=s0+sigi}), for some sos polynomials $s_0,s_1,\ldots,s_m$ of high enough degree (see also~\cite{nie2007Putinarcomplexity} for degree bounds). Even with absolutely no qualifications about the set $\mathcal{S}$, there are other Positivstellensatz theorems (e.g., due to Stengle~\cite{stengle1974nullstellensatz}) that certify nonnegativity of a polynomial on a basic semialgebraic set using sos polynomials. These certificates are only slightly more complicated than (\ref{eq:p=s0+sigi}) and involve sos multipliers associated with products among polynomials $g_i$ that define $\mathcal{S}$~\cite{sdprelax}. A great reference for the interested reader is the survey paper by Laurent~\cite{Laurent_survey}.

The computational advantage of a certificate of (global or local) nonnegativity via sum of squares polynomials is that it can be automatically found by semidefinite programming. What establishes the link between sos polynomials and SDP is the following well-known theorem. Recall that a symmetric $n \times n$ matrix $A$ is \emph{positive semidefinite} (psd) if $x^TAx\geq 0, \forall x\in\mathbb{R}^n$, and that semidefinite programming is the problem of optimizing over psd matrices subject to affine inequalities on their entries~\cite{VaB:96}. We denote the positive semidefiniteness of a matrix $A$ with the standard notation $A\succeq 0$.

% there are Positivstellensatz theorems from real algebraic geometry (see Section * and *,*) that allow us to convert this problem into one of asking that certain new polynomials, depending on $p,g_1,\ldots,g_m$ are sums of squares. In either case, what enables a reformulation of the ``sos constraints'' as semidefinite programming constraints is the following well-known theorem.

\begin{theorem}[see, e.g., \cite{PhD:Parrilo},\cite{sdprelax}]
\label{thm:sos.sdp}
A multivariate polynomial $p(x)$ in $n$ variables and of degree
$2d$ is a sum~of~squares if and only if there exists a symmetric matrix $Q$ (often called the Gram matrix) such that
\begin{equation}\label{eq:p=z'Qz}
\begin{array}{rll}
p(x)&=&z^{T}Qz, \\
Q&\succeq&0,
\end{array}
\end{equation}
where $z$ is the vector of monomials of degree up to $d$
\begin{equation*}\label{eq:monomials}
z=[1,x_{1},x_{2},\ldots,x_{n},x_{1}x_{2},\ldots,x_{n}^d].
\end{equation*}
\end{theorem}

The search for the matrix $Q$ satisfying a positive semidefiniteness constraint, as well as linear equality constraints coming from (\ref{eq:p=z'Qz}) is a semidefinite programming problem. The size of the matrix $Q$ in this theorem is $${n+d\choose d}\times {n+d\choose d},$$
which approximately equals $n^d \times n^d$. While this number is polynomial in $n$ for fixed $d$, it can grow rather quickly even for low degree polynomials. For example, the polynomials that we will be requiring to be sos in our controller design problem for the quadrotor (Section~\ref{sec:quadrotor}) have 16 variables and degree 6 and result in a Gram matrices with about half a million decision variables. A semidefinite constraint of this size is quite expensive---for example, the SDP solvers of SeDuMi~\cite{sedumi} and MOSEK~\cite{mosek} fail to solve the quadrotor problem on our machine and quickly run out of memory.

%For example, a degree-$4$ polynomial ($d=2$) in $50$ variables has 316251 coefficients and its Gram matrix, which would need to be positive semidefinite, consists of 879801 decision variables  **[check these numbers]**. A semidefinite constraint of this size is extremely expensive, and in fact a problem with $50$ variables is far beyond the current real of possibilities in SOS optimization. In absence of problem structure, sum of squares problems involving degree $4$ or $6$ polynomials are currently limited, roughly speaking, to a handful to a dozen of variables.

\subsection{DSOS and SDSOS Optimization}
\label{sec:dsos_and_sdsos}

In order to address the problem of scalability posed by SDP, we have recently introduced \cite{Ahmadi14,dsos_ciss14} alternatives to SOS programming that lead to linear programs (LPs) and second order cone programs (SOCPs). The key insight there is to replace the condition that the Gram matrix $Q$ be positive semidefinite (psd) with stronger sufficient conditions in order to obtain inner approximations to the cone $SOS_{n,d}$ of sos polynomials in $n$ variables and of degree $d$. In particular, $Q$ will be required to be either \emph{diagonally dominant} (dd) or \emph{scaled diagonally dominant} (sdd). We recall these definitions below.

\begin{definition}
\label{def:dd}
A symmetric matrix $A$ is \emph{diagonally dominant} (dd) if $a_{ii} \geq \sum_{j \neq i} |a_{ij}|$ for all $i$.
\end{definition}

We will refer to the set of $n \times n$ dd matrices as $DD_n$.

\begin{remark}
\label{rmk:dd_lp}
It is clear from Definition \ref{def:dd} that the set $DD_n$ has a polytopic description and can thus be optimized over using LP.
\end{remark}

%\begin{definition}
%A symmetric matrix $A$ is \emph{scaled diagonally dominant} (sdd) if there exists an element-wise positive vector $y$ such that:
%$$a_{ii} y_i \geq \sum_{j \neq i} |a_{ij}| y_j, \forall i.$$
%Equivalently, $A$ is sdd if there exists a positive diagonal matrix $D$ such that $AD$ (or equivalently, $DAD$) is diagonally dominant.
%\end{definition}

\begin{definition}
\label{def:sdd}
Denote the set of $n \times n$ symmetric matrices as $S^n$. Let $M_{2 \times 2}^{ij} \in S^n$ denote the symmetric matrix with all entries zero except the elements $M_{ii}, M_{ij}, M_{ji}, M_{jj}$. Then, a symmetric matrix $A$ is \emph{scaled diagonally dominant} (sdd) if it can be expressed in the following form:
$$A = \sum_{i\neq j} M_{2 \times 2}^{ij},  \quad \begin{bmatrix} M_{ii} & M_{ij}  \\ M_{ji} & M_{jj} \end{bmatrix} \succeq 0.$$
\end{definition}

\begin{remark}
The relationship between dd and sdd matrices is made clear in \cite{Ahmadi14}. As we show there, a symmetric matrix $A$ is sdd if and only if there exists a positive diagonal matrix $D$ such that $AD$ (or equivalently, $DAD$) is diagonally dominant.
\end{remark}

The set of $n \times n$ sdd matrices will be denoted by $SDD_n$. We note that sdd matrices are sometimes referred to as \emph{generalized diagonally dominant} matrices \cite{Boman05}. 

\begin{theorem}
\label{thm:sdd_socp}
The set of matrices $SDD_n$ can be optimized over using second order cone programming.
\end{theorem}
\begin{proof}
Positive semidefiniteness of the $2 \times 2$ matrices in Definition \ref{def:sdd} is equivalent to the diagonal elements $M_{ii}, M_{jj}$, along with the determinant $M_{ii}M_{jj} - M_{ij}^2$, being nonnegative. This is a \emph{rotated quadratic cone} constraint and can be imposed using SOCP~\cite{Alizadeh03}. 
\end{proof}

%\begin{remark}
%The fact that diagonal dominance is a sufficient condition for positive semidefiniteness follows directly from Gershgorin's circle theorem. This also implies that sdd matrices are psd since the eigenvalues of $DAD$ have the same sign as those of $A$ when $D$ is a  diagonal matrix with positive entries. Hence, denoting the set of $n \times n$ symmetric positive semidefinite matrices (psd) as $S_n^+$, we have from the definitions above that:
%$$DD_n \subseteq SDD_n \subseteq S_n^+.$$
%\end{remark}

\begin{remark}
The fact that diagonal dominance is a sufficient condition for positive semidefiniteness follows directly from Gershgorin's circle theorem. The fact that sdd implies psd is immediate from Definition \ref{def:sdd} since a sdd matrix is a sum of psd matrices. Hence, denoting the set of $n \times n$ symmetric positive semidefinite matrices (psd) as $S_n^+$, we have from the definitions above that:
$$DD_n \subseteq SDD_n \subseteq S_n^+.$$
\end{remark}

%\subsubsection{DSOS and SDSOS Polynomials}
%\label{subsubsec:dsos_and_sdsos}

We now introduce some naturally motivated cones that are inner approximations of the cone of nonnegative polynomials and that lend themselves to LP and SOCP. In analogy with the representation of sos polynomials in terms of psd matrices (Theorem \ref{thm:sos.sdp}), we define the \emph{dsos} and \emph{sdsos} polynomials in terms of dd and sdd matrices respectively.

\begin{definition}[\cite{Ahmadi14,dsos_ciss14}] \label{def:dsos.sdsos.rdsos.rsdsos}
\quad
\begin{itemize}
\item A polynomial $p$ of degree $2d$ is \emph{diagonally-dominant-sum-of-squares} (dsos) if it admits a representation as $p(x)=z^T(x)Qz(x)$, where $z(x)$ is the standard monomial vector of degree $d$, and $Q$ is a dd matrix. \\
\item A polynomial $p$ of degree $2d$ is \emph{scaled-diagonally-dominant-sum-of-squares} (sdsos) if it admits a representation as $p(x)=z^T(x)Qz(x)$, where $z(x)$ is the standard monomial vector of degree $d$, and $Q$ is a sdd matrix.
\end{itemize}

We denote the set of polynomials in $n$ variables and degree $d$ that are dsos and sdsos by $DSOS_{n,d}$ and $SDSOS_{n,d}$ respectively.
\end{definition}
The following inclusion relations are straightforward: $$DSOS_{n,d}\subseteq SDSOS_{n,d}\subseteq SOS_{n,d}.$$

\begin{theorem} \ 
The set $DSOS_{n,d}$ is polyhedral and the set $SDSOS_{n,d}$ has a second order cone representation. For any fixed $d$, optimization over $DSOS_{n,d}$ (resp. $SDSOS_{n,d}$) can be done with linear programming (resp. second order cone programming), of size polynomial in $n$.
\end{theorem}
\begin{proof}
This follows directly from Remark \ref{rmk:dd_lp} and Theorem \ref{thm:sdd_socp}. The size of these programs is polynomial in $n$ since the size of the Gram matrix is ${n+d \choose d}\times {n+d \choose d}$, which scales as $n^d$.
\end{proof}

\begin{remark}
While here we have chosen to define the $DSOS_{n,d}$ and $SDSOS_{n,d}$ cones  directly in terms of dd and sdd matrices in order to expose their LP and SOCP characterizations, it is more natural to define them as sos polynomials of a \emph{particular form}. This alternate characterization is provided in \cite{Ahmadi14}. In particular, we have the following equivalent definitions:
\begin{itemize}
\item A polynomial $p$ is dsos if it can be written as 
$$p=\sum_i \alpha_i m_i^2 + \sum_{i,j} \beta_{ij}^+ (m_i +  m_j)^2 +  \beta_{ij}^- (m_i -  m_j)^2,$$
for some monomials $m_i, m_j$ and some constants $\alpha_i,\beta_{ij}^+,\beta_{ij}^- \geq 0$. \\

\item A polynomial $p$ is sdsos if it can be written as 
$$p=\sum_i \alpha_i m_i^2 + \sum_{i,j}  (\beta_i^+ m_i + \gamma_j^+ m_j)^2 + (\beta_i^- m_i - \gamma_j^- m_j)^2,$$
for some monomials $m_i, m_j$ and some constants $\alpha_i,\beta_i^+,\gamma_j^+,\beta_i^-,\gamma_j^- \geq 0$.
\end{itemize}
\end{remark}

We will refer to optimization problems with a linear objective posed over the cones $DSOS_{n,d}$, $SDSOS_{n,d}$, and $SOS_{n,d}$ as DSOS programs, SDSOS programs, and SOS programs respectively. In general, quality of approximation decreases, while scalability increases, as we go from SOS to SDSOS to DSOS programs. Depending on the size of the application at hand, one may choose one approach over the other. In this paper, we will be using SOS optimization (Section~\ref{sec:jamming}) and SDSOS optimization (Sections~\ref{sec:barriers} and~\ref{sec:quadrotor}) in our numerical experiments. The reader is referred to~\cite{dsos_cdc14,dsos_ciss14,Ahmadi14} for many numerical examples involving DSOS optimization. We also remark in passing that SDSOS or even DSOS programming enjoy many of the same theoretical (asymptotic) guarantees of SOS programming---results of this nature are proven in~\cite{Ahmadi14}.

We now proceed to some potential operations research applications of the tools discussed so far.

\section{Wireless coverage with minimum transmission}\label{sec:jamming} %In adversarial settings, it is of great value to degrade the communication capability of the enemy. The problem described in this section is motivated by some interesting and relatively recent work in~\cite{commander2007wireless},~\cite{Jamming_journal}, which also appears as a thesis in~\cite{Jamming_phd}. The goal is to use a number of jamming devices, which transmit electromagnetic waves, to suppress the communication network of the enemy. In the latter reference, motivation for this problem is also drawn from the fact that improvised explosive devices (IEDs) are almost always detonated by some form of radio frequency device such as cellular telephones, pagers, and garage door openers. 

In the problem considered in this section we have a number $n$ of wireless electromagnetic transmitters located at positions $(\bar{x}_i,\bar{y}_i), i=1,\ldots,n$ on the plane. Each transmitter is an omnidirectional power source, emitting waves in all directions with equal intensity. Due to the laws of electromagnetics, the energy $E_i$ propagated from each jamming device is inversely proportional to the squared distance from the device:
$$E_i(x,y)=\frac{c_i\lambda}{(x-\bar{x}_i)^2+(y-\bar{y}_i)^2},$$
where $\lambda$ is some propagation constant, set hereafter to $1$ with no loss of generality, and $c_i$ is the transmission rate of device $i$. The goal is to make sure that certain regions of the plane are guaranteed to receive a given cumulative energy level of at least $C$ units, while minimizing transmission power. These regions can for example be populated urban geographical domains where a wireless service provider would like to guarantee a certain level of signal quality.

%The coverage goal is to be achieved by the minimum transmission rate possible and we may also be subject to upper bounds on the transmission rate of each device: $c_i\leq\gamma_i$. 

%(either minimizing $\sum_{i=1}^n c_i$ or $\max_i c_i$)

The problem we describe is motivated by some interesting and relatively recent work in~\cite{commander2007wireless},~\cite{Jamming_journal} (see also the thesis~\cite{Jamming_phd}), where the motivation is instead to jam the communication network of an adversary with a wireless transmitter. We note, however, that there are a few differences between our setting and that of~\cite{commander2007wireless} and~\cite{Jamming_journal}, the main one being the assumption about the region to be covered. Reference~\cite{commander2007wireless} assumes that this region is a set of isolated points (the location of the adversary is known) and this results in a simplified problem. However, more complex objectives are considered by the authors; e.g., the goal is to make the communication graph of the enemy disconnected, or to jam a prescribed fraction of the enemy locations, or to decide which transmitters to turn off. On the opposite end, the work in~\cite{Jamming_journal} assumes absolutely nothing about the location of the adversary. As a result, the goal is to cover an entire rectangular region by a prescribed level of jamming power. 
%While this is an appropriate assumption in settings with very limited information, in general the amount of transmitted power may be much larger than it needs to be.
Our setting, by contrast, allows for the region to be covered to be the union of arbitrary basic semialgebraic sets (see, e.g., Figure~\ref{subfig:input}); this obviously enhances the modeling power. We should also comment that neither our work, nor the works in~\cite{commander2007wireless} and~\cite{Jamming_journal}, satisfactorily address the more difficult problem of optimizing over the location of the transmitters.

A formal summary of our setting is as follows. We are given as input the following quantities: $C$ (required coverage level), $\gamma_i$ (upper bounds on transmission rates), $(\bar{x}_i,\bar{y}_i), i=1,\ldots,n$ (location of our transmitters), $\mathcal{B}_j, j=1,\ldots,m$ (basic semialgebraic sets describing regions to be covered). We assume that the transmitters are outside of the location sets $\mathcal{B}_j$. The goal is to find transmission rates $c_i$ to solve the following optimization problem\footnote{An alternative reasonable objective is to minimize $\max_i c_i$. This can as easily be handled.}:

%such that we minimize either $\sum_{i=1}^n c_i$ or $\max_i c_i$ subject to the following constraints:

\begin{equation}\label{eq:wireless.coverage.problem}
\begin{array}{rlll}
\mbox{minimize} &\ &\sum_{i=1}^n c_i& \   \\
 c_i&\leq&\gamma_i, &\forall i=1,\ldots,n, \\
E(x,y)\mathrel{\mathop:}=\sum_{i=1}^n \frac{c_i}{(x-\bar{x}_i)^2+(y-\bar{y}_i)^2} &\geq& C,& \forall (x,y)\in\mathcal{B}_j, j=1,\ldots,m.
\end{array}
\end{equation}

Note that the latter constraints are requiring certain rational functions to be nonnegative on certain basic semialgebraic sets. Upon taking common denominators, we can rewrite these constraints as polynomial inequality constraints:

\begin{equation}\label{eq:after.common.denom}
\begin{array}{rr}
p(x,y)&\mathrel{\mathop:}=-C\prod_{i=1}^{n}[(x-\bar{x}_i)^2+(y-\bar{y}_i)^2]+\sum_{i=1}^n c_i\prod_{k\neq i}[(x-\bar{x}_k)^2+(y-\bar{y}_k)^2]\geq 0, \\
\ &\forall (x,y)\in\mathcal{B}_j, j=1,\ldots,m.
\end{array}
\end{equation}

Observe that the degree of the polynomial $p(x,y)$ is two times the number of transmitters. Since we are dealing with polynomial inequalities in only two variables, we have no scalability issues restraining us from applying the sos relaxation. Let each set $\mathcal{B}_j$ be defined as $$\mathcal{B}_j=\{x| \quad g_{j,1}(x,y)\geq 0,\ldots, g_{j,k_j}(x,y)\geq 0\},$$ for some bivariate polynomials $g_{j,1},\ldots, g_{j,k_j}$. The optimization problem that we will be solving is:

\begin{equation}\label{eq:wireless.sos.relaxation}
\begin{array}{rlll}
\mbox{minimize} &\ &\sum_{i=1}^n c_i& \  \\
p&=&\sigma_0+\sum_{i=1}^{k_j}\sigma_{j,k}g_{j,i}, &j=1,\ldots,m, \\
\ &\ &\sigma_0,\sigma_{j,k}& \mbox{sos},
\end{array}
\end{equation}
where $p$ is as in (\ref{eq:after.common.denom}) and $\sigma_0,\sigma_{j,k}$ are bivariate polynomials whose degree is upper bounded by some even integer $d$. Note that the above is a semidefinite programming problem (via Theorem~\ref{thm:sos.sdp}) with decision variables consisting of the scalars $c_i$ and the coefficients of the polynomials $\sigma_0,\sigma_{j,k}$. It is easy to see that for each value of the degree $d$, the optimal value of (\ref{eq:wireless.sos.relaxation}) is an upper bound on the optimal value of (\ref{eq:wireless.coverage.problem}). Moreover, since in our setting each set $\mathcal{B}_i$ satisfies the Archimedean property\footnote{Indeed each set $\mathcal{B}_i$ is compact and the entire environment can be placed in a ball of some prescribed radius $R$. This quadratic constraint can be added to the description of each $\mathcal{B}_i$ to satisfy the Archimedean property.}, Putinar's Positivstellensatz tells us that by increasing $d$, we will be able to solve (\ref{eq:wireless.coverage.problem}) to global optimality.

Let us now solve a concrete example. Our input data is demonstrated in Figure~\ref{subfig:input}. We have two transmitters, located at points $(1,1.5)$ (called transmitter 1) and $(2,1)$ (called transmitter 2) on the plane. The area to be covered is given by the five ellipsoidal regions $$\mathcal{B}_j=\{z\mathrel{\mathop:}=(x,y)^T|\ (z-z_j)^TA_j(z-z_j)\leq \alpha_j\},$$ with $A_1=\begin{bmatrix}
3 &1 \\ 1 &1 
\end{bmatrix}, A_2=\begin{bmatrix}
1 &0 \\ 0 &3 
\end{bmatrix},A_3=\begin{bmatrix}
1 &0 \\ 0 &1 
\end{bmatrix},A_4=\begin{bmatrix}
1 &-1 \\ -1 &3 
\end{bmatrix},A_5=\begin{bmatrix}
5 &0 \\ 0 &1 
\end{bmatrix}, z_1=(1.1,1.75)^T, z_2=(1.25,2)^T, z_3=(1.5,1.75)^T, z_4=(1.8,1.8)^T, z_5=(2,1.4)^T, \alpha_1=\alpha_2=\alpha_3=\alpha_4=0.1, \alpha_5=0.2.$

%\begin{figure}
%  \begin{center}
%    \includegraphics[width=.5\columnwidth]{Figures/input.png}
%  \end{center}
%\end{figure}

\begin{figure}%[thpb]
\begin{center}
    \mbox{
      \subfigure[The input to the problem.]
      {\label{subfig:input}\scalebox{0.22}{\includegraphics{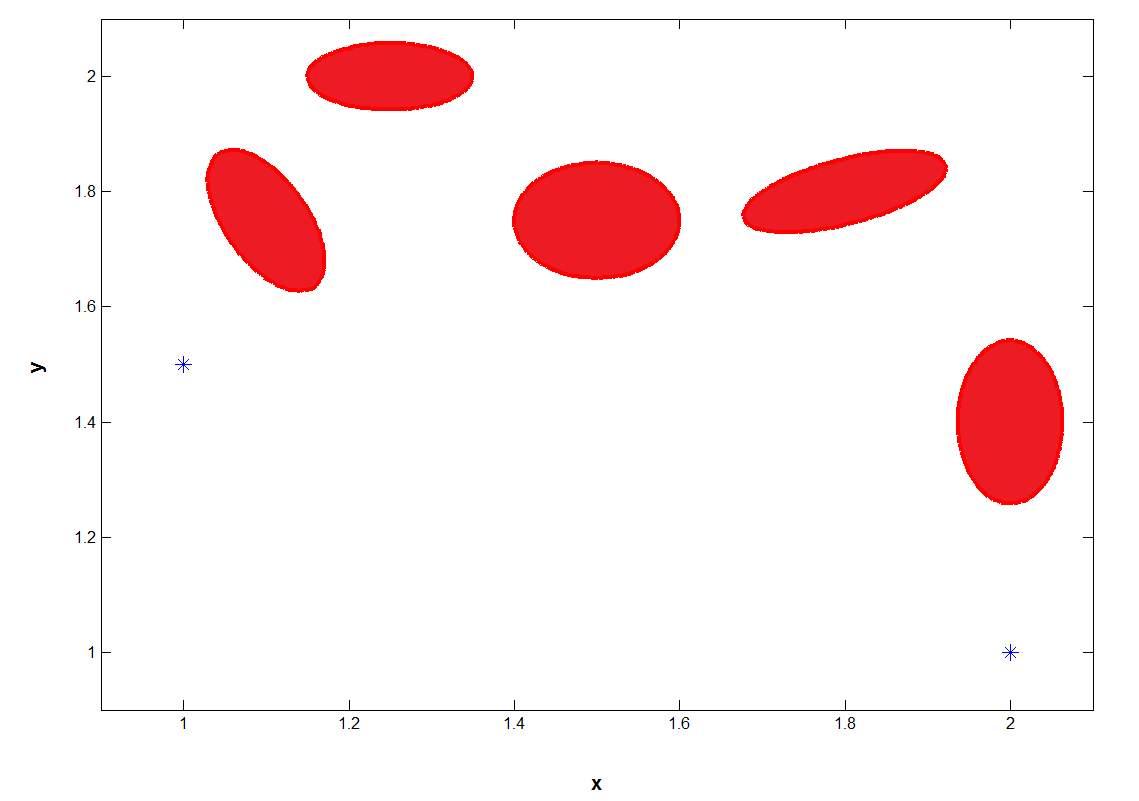}}}}
\mbox{
      \subfigure[Log of the cumulative energy at the SDP solution.]
      {\label{subfig:log.output}\scalebox{0.22}{\includegraphics{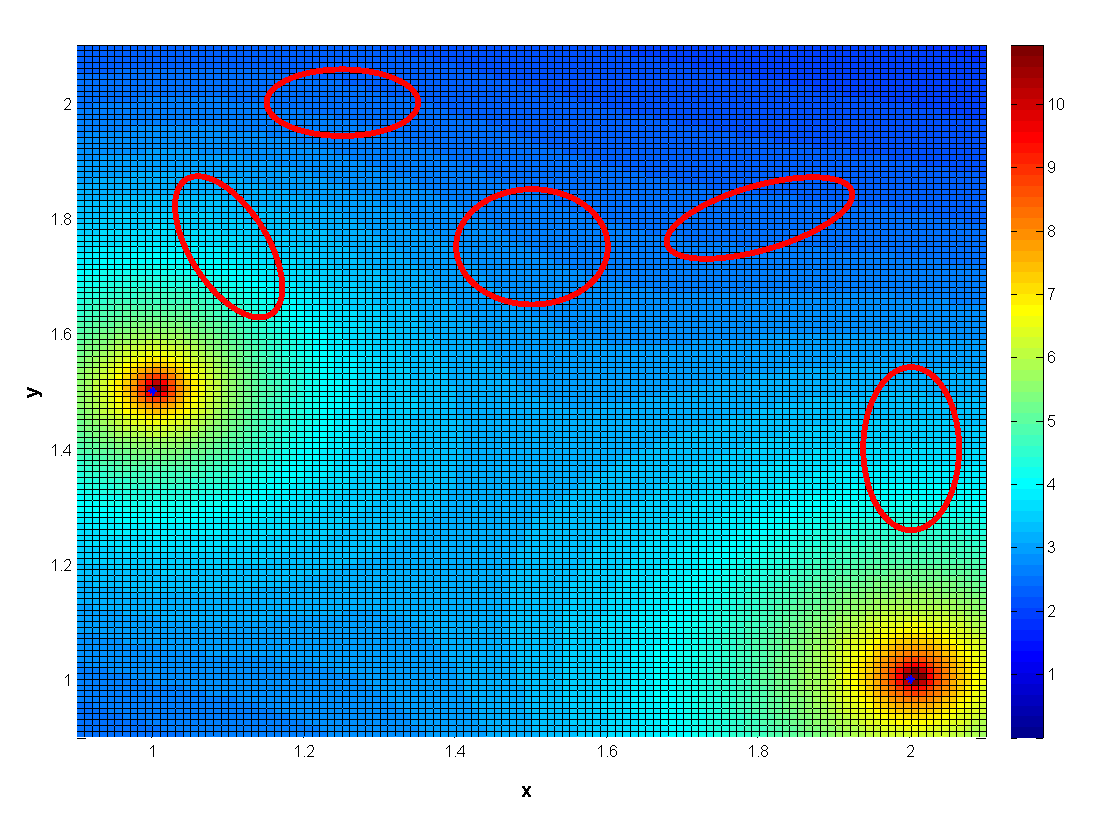}}}

      \subfigure[Region that receives a total energy of at least $C=10$ units (in dark red).]
      {\label{subfig:region.covered}\scalebox{0.22}{\includegraphics{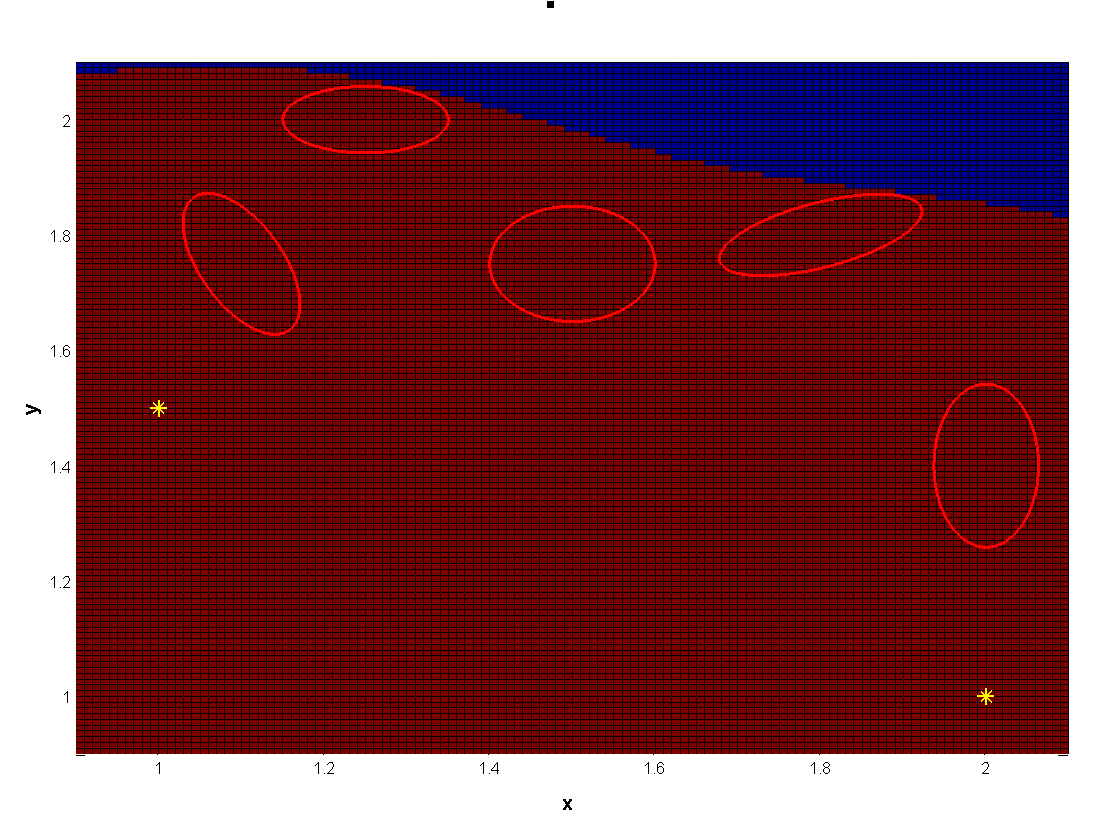}}} }

    \caption{An instance of the wireless coverage problem.}
\label{fig:robust.jamming.problem}
\end{center}
\vspace{-20pt}
\end{figure}

The required energy level on these areas is $C=10$ and the upper bounds on both transmission rates $c_1,c_2$ is $11$. We first would like to know if by only turning on one of the two transmitters we can meet the required energy level. For the transmitter at the location $(2,1)$, the optimal value of the SDP in (\ref{eq:wireless.sos.relaxation}) with degree of sos multipliers set to zero (i.e., constant multipliers) is $17.594$. In fact, in this case, we know that this upper bound is already exact! This is because in the case of one transmitter, the polynomial $p$ in (\ref{eq:after.common.denom}) is quadratic. If a quadratic polynomial is nonnegative on a region defined by another quadratic, this fact is always certified by a constant degree multiplier---this is the celebrated $\mathcal{S}$-lemma; see~\cite{S_lemma_survey}. Similarly, if we solve the problem for the transmitter located at $(1,1.5)$, the optimal value of (\ref{eq:wireless.sos.relaxation}) which matches the optimal value of (\ref{eq:wireless.coverage.problem}) is 11.446. So our task is indeed not achievable with one transmitter only.

With both transmitters on, the SDP in (\ref{eq:wireless.sos.relaxation}) is infeasible for degree-0 sos multipliers (giving an upper bound of infinity). However, when we increase the degree of these multipliers to $2$, a solution is returned with $c_1=2.561$ and $c_2=5.550$ at optimality. By further increasing the degree of our sos multipliers, no improvement in optimal value is observed and we conjecture that the numbers above are already optimal for the original problem (\ref{eq:wireless.coverage.problem}). Figure~\ref{subfig:log.output} shows the logarithm of the cumulative energy level $E(x,y)$ at each point in space. (The logarithm is taken to better observe the dispersion of energy.) Figure~\ref{subfig:region.covered} shows all pixels that receive the required energy level of $C=10$ units. As promised, all five ellipsoids are covered and interestingly the boundary of the region covered touches two of the ellipsoids.

%In its simplest form, the problem is to find the optimal location $x_i, i=1,\ldots,n$ of a set of $n$ jamming devices to disable the communication capability of some targets whose positions are believed to be among $N$ locations $\hat{x}_j, j=1:\ldots,N$. 
%The \emph{jamming effectiveness} of each jamming device, being an electromagnetic transmitter, is assumed to be inversely proportional to the squared distance from the device. For each target to be disabled, a total intensity of $C$ units needs to be received at its location from all jamming devices. This leads to the nonconvex set of constraints:
%$$\sum_{i=1}^n\frac{1}{||\hat{x}_j-x_i||^2}\geq C,\quad \forall j=1,\ldots,N.$$ Upon taking common denominators and rearranging terms, this is a polynomial optimization problem. A natural additional constraint that can be handled within POP is for the jamming devices to be more than a certain distance $r$ away from the targets, to avoid exposure. A reasonable objective may be to minimize $n$, the number jamming devices, as they may be expensive or limited. Algebraic techniques are very powerful for providing certificate of optimality here. Given a feasible solution with $k$ jamming devices, algebraic optimization algorithms can provide a short certificate showing that a solution with $k-1$ jamming devices is impossible.

\section{Lyapunov theory and optimization} \label{sec:Lyap.background}

\begin{figure}[h]
\centering
\includegraphics[width=.85\columnwidth]{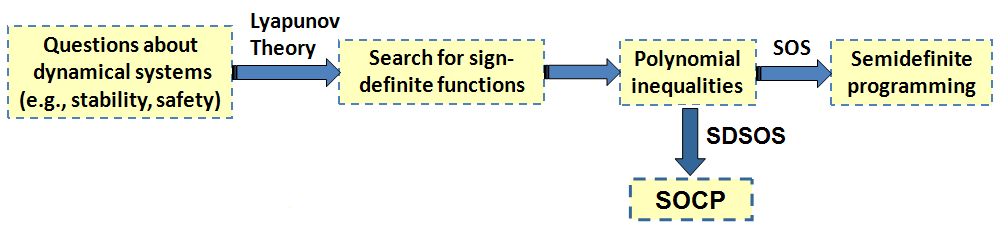}
\caption{{\small The steps involved in Lyapunov analysis of dynamical systems via convex optimization.}} \label{fig:lyap2sdp}
\end{figure}

The examples presented in our next two sections involve decision-making about trajectories of dynamical systems. The machinery that allows us to reduce such tasks to problems in optimization is \emph{Lyapunov theory}. As depicted in Figure~\ref{fig:lyap2sdp}, the general idea is the following: In order to guarantee that trajectories of dynamical systems satisfy certain desired properties, it will be enough to find certain scalar valued functions that satisfy certain inequalities. These functions will be parameterized as polynomials and DSOS/SDSOS/SOS relaxation techniques will be used to find their unknown coefficients in such a way that the desired inequalities are automatically satisfied. For our two applications, we explain next what these inequalities actually are.

%the machinery that allows for a reformulation of questions about dynamical systems as 
%optimization problems searching for functions satisfying certain inequalities is \emph{Lyapunov theory}. We briefly explain this powerful idea in the context of the two applications that we will be presenting.

%Numerous fundamental problems in nonlinear dynamics and control, such as stability, invariance, robustness, collision avoidance, controller synthesis, etc., can be turned by means of ``Lyapunov theorems'' into problems about finding special functions---the so-called \emph{Lyapunov functions}---that satisfy certain sign conditions. The task of constructing Lyapunov functions has traditionally been one of the most fundamental and challenging tasks in control. In recent years, however, advances in convex programming and in particular in the theory of sum of squares optimization have allowed for the search for Lyapuonv functions to become \emph{fully automated}. Figure~\ref{fig:lyap2sdp} summarizes the steps involved in this process.

\paragraph{Barrier functions (Section~\ref{sec:barriers}).} Consider a differential equation $\dot{x}=f(x)$, where $\dot{x}$ denotes the derivative of the state vector $x$ with respect to time and $f:\mathbb{R}^n\rightarrow\mathbb{R}^n$ is a polynomial function. Suppose we are given two basic semialgebraic sets $\mathcal{S}_{\mbox{safe}}$ and $\mathcal{S}_{\mbox{unsafe}}$ and we want to guarantee that trajectories starting in $\mathcal{S}_{\mbox{safe}}$ would never end up in $\mathcal{S}_{\mbox{unsafe}}$. This guarantee can be achieved if we succeed in finding a function $V:\mathbb{R}^n\rightarrow\mathbb{R}$, called a \emph{barrier function}~\cite{barrier_certificates_PrajnaJadbabaie},~\cite{barrier_certificates_TAC_PrajnaJadbabaie},~\cite{barry2012safety}, with the following three properties:
$$V(x)<1 \quad \forall x\in\mathcal{S}_{\mbox{safe}}, \quad V>1 \quad \forall x\in \mathcal{S}_{\mbox{unsafe}}, \quad \dot{V}(x)\leq0 \quad \forall x.$$ 
The expression $\dot{V}$ denotes the time derivative of $V$ along trajectories. If $V$ is a polynomial, $\dot{V}$ will also be a polynomial given (via the chain rule) by: $$\dot{V}(x)=\langle\nabla V(x),f(x)\rangle.$$
The three inequalities above imply that it is impossible for a trajectory to go from $\mathcal{S}_{\mbox{safe}}$ to $\mathcal{S}_{\mbox{unsafe}}$ since the function $V$ evaluated on this trajectory would need to go from a value less than one to a value more than one, but that cannot happen since the value of $V$ is non-increasing along trajectories.

\paragraph{Stability and region of attraction computation (Section~\ref{sec:quadrotor}).} Suppose once again that we have a differential equation $\dot{x}=f(x)$ with origin as an equilibrium point (i.e., satisfying $f(0)=0$). In numerous applications in control and robotics, one would like to make sure that deviations from an equilibrium point tend back to the equilibrium point. This is the notion of asymptotic stability. A particularly important problem in this area is the so-called ``region of attraction (ROA) problem'': For what set of initial conditions in $\mathbb{R}^n$ do trajectories flow to the origin? This question can be addressed with Lyapunov theory. In fact, Lyapunov's stability theorem (see, e.g.,~\cite[Chap. 4]{Khalil:3rd.Ed}) tells us that if we can find a (Lyapunov) function $V:\mathbb{R}^n\rightarrow\mathbb{R}$, which together with its gradient $\nabla V$ satisfies

\begin{equation}\label{eq:poly_Lyap_inequalities}
V(x)>0\quad \forall x\neq0,  \quad \quad \mbox{and} \quad \quad \dot{V}(x)=\langle\nabla V(x),f(x)\rangle<0\quad \forall x\in\{x|\ V(x)\leq\beta, x\neq 0\},
\end{equation}
then the sublevel set $\{x|\ V(x)\leq\beta\}$ is part of the region of attraction. Notice again that if $f$ is a polynomial function (an immensely important case in applications~\cite[Chap. 4]{AAA_PhD}), and if we parameterize $V$ as a polynomial function, then the search for the coefficients of $V$ satisfying the conditions in (\ref{eq:poly_Lyap_inequalities}) is an optimization problem over the set of nonnegative polynomials.

\section{Real-time Planning with Barrier Functions}
\label{sec:barriers}

\begin{figure}[H]
\centering
\includegraphics[trim = 0mm 0mm 0mm 0mm, clip, width=.12\textwidth]{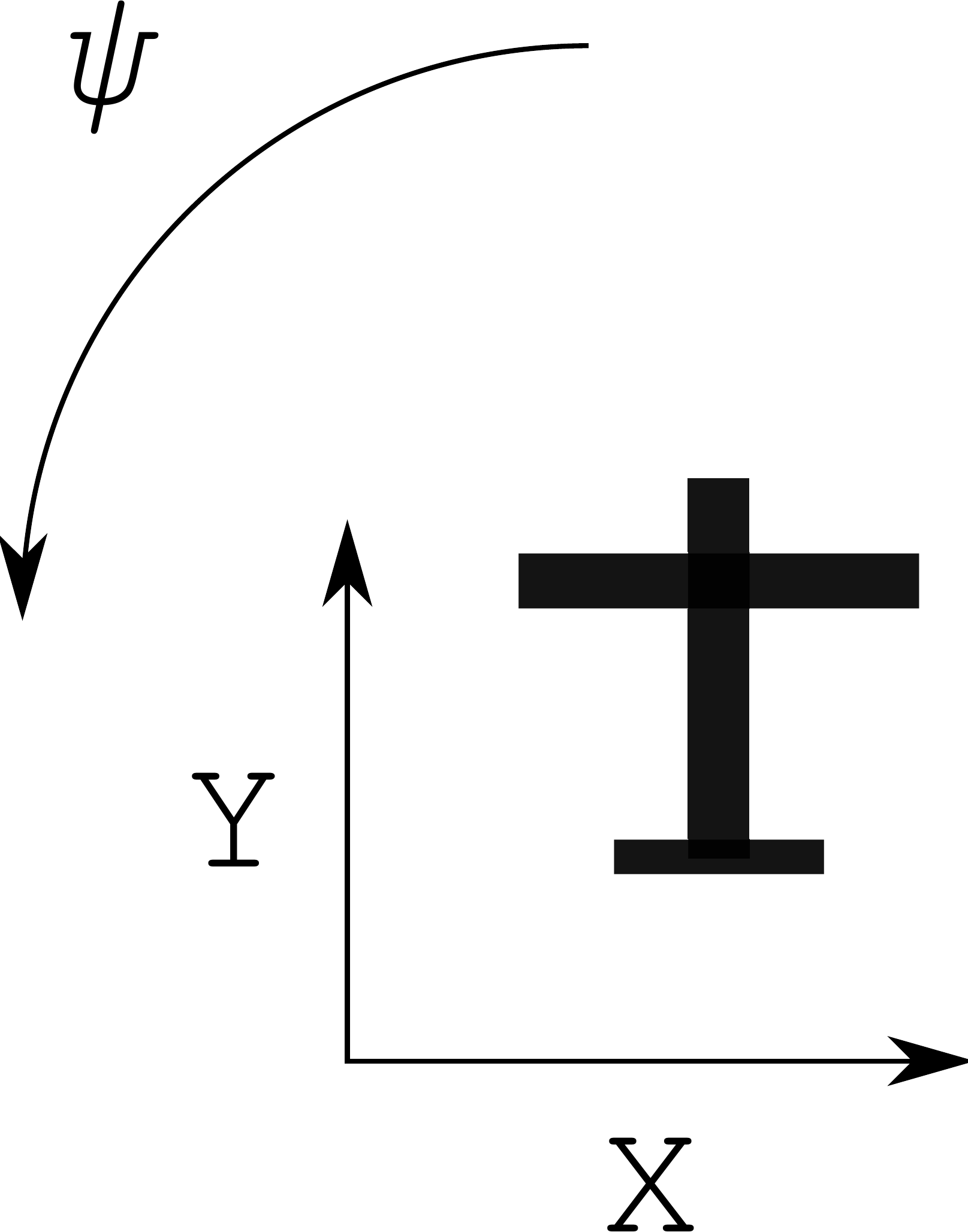}
\caption{An illustration of the states of the UAV model we consider.}
\label{fig:airplane_model}
\vspace{-10pt}
\end{figure}

One promising application domain for polynomial optimization in transportation is for real-time planning and control on autonomous vehicles. In this example, we consider such an application for a simple model of an unmanned aerial vehicle (UAV) navigating through a cluttered two dimensional environment. In order to make the navigation task more realistic, we also consider a bounded but uncertain ``cross-wind'' term in the dynamics. This results in an uncertain differential equation and requires reasoning about \emph{families} of trajectories that the system could end up following, making the problem more challenging. The states and dynamics of the UAV are inspired by the widely-used Dubins car model \cite{Dubins57} and are given by:

\begin{equation}
\bf{x} = \left[ \begin{array}{c}
x \\
y \\
\psi \end{array} \right],
\qquad
\dot{\bf{x}} = f({\bf{x}},u,w) = \left[ \begin{array}{c}
\dot x \\
\dot y \\
\dot \psi \end{array} \right] = \left[ \begin{array}{c}
-v \sin \psi + w \\
v \cos \psi \\
u \end{array} \right],
\label{eq:airplane_dynamics}
\end{equation}
where $x$ and $y$ are the x and y positions of the UAV in the environment, $v = 1$ m/s is the speed of the airplane, $\psi$ is the yaw angle, $u$ is the control input and $w$ is the ``cross-wind''  (bounded between $[-0.05, 0.05]$).  An illustration of the states of the model are given in Figure \ref{fig:airplane_model}. We Taylor expand these dynamics to degree $3$ to obtain polynomial dynamics in order to use DSOS/SDSOS/SOS programming.

\begin{figure}[H]
\centering
\includegraphics[trim =5mm 45mm 0mm 50mm, clip, width=.25\textwidth]{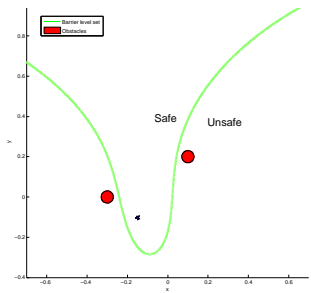}
\caption{A barrier function computed for a particular initial state and obstacle configuration. The UAV is guaranteed to remain safe when the controller is executed despite the effects of the cross-wind. The green curve is a level set of a degree-4 polynomial found by SDSOS optimization.}
\label{fig:barrier_plot}
\vspace{-20pt}
\end{figure}

\begin{figure}[H]
\centering
\subfigure[Environment 1]{%
\includegraphics[trim = 180mm 30mm 160mm 20mm, clip, width=.25\textwidth]{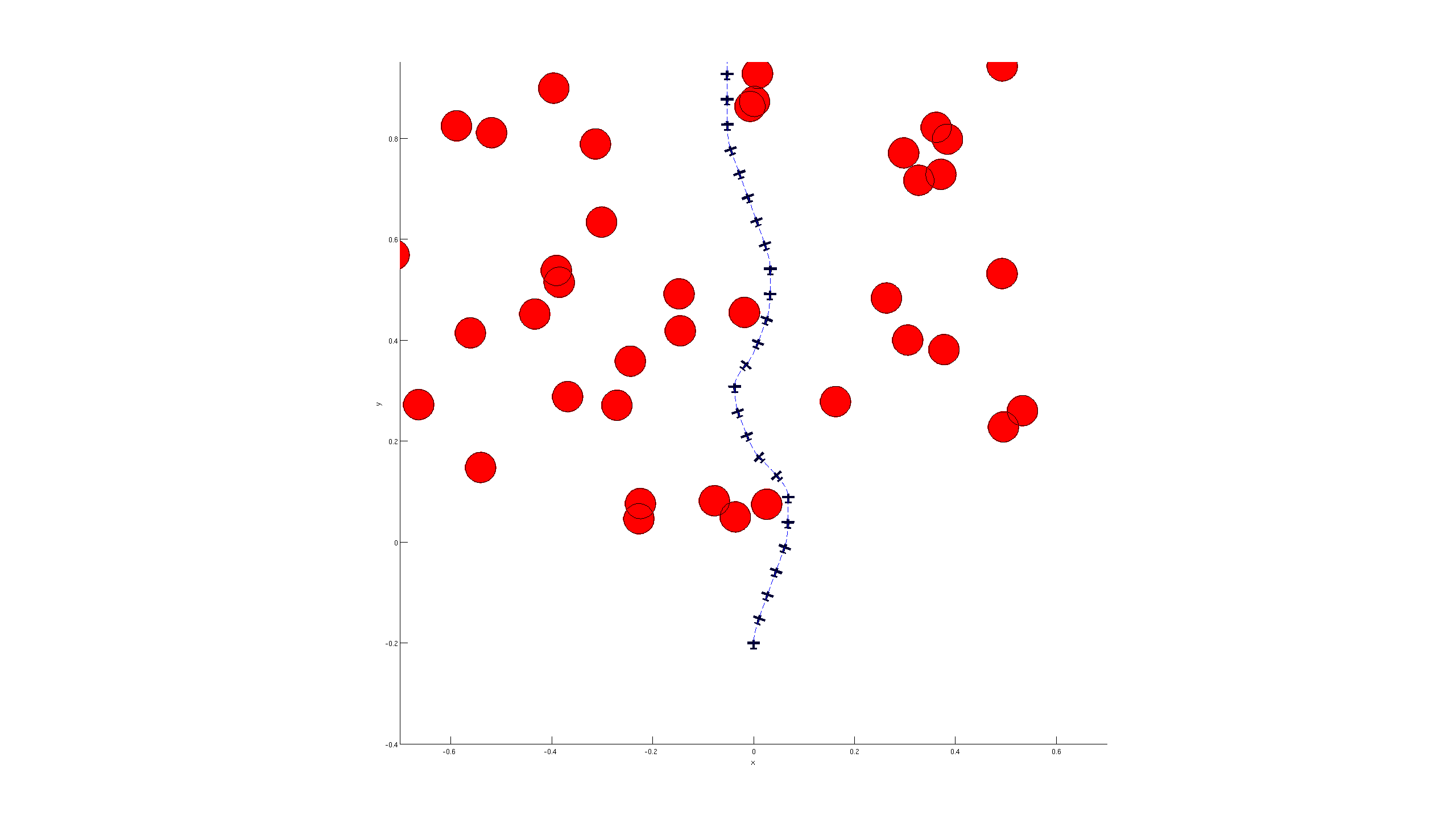}
\label{fig:env1}}
\hfill
\subfigure[Environment 2]{%
\includegraphics[trim = 180mm 30mm 160mm 20mm, clip, width=.25\textwidth]{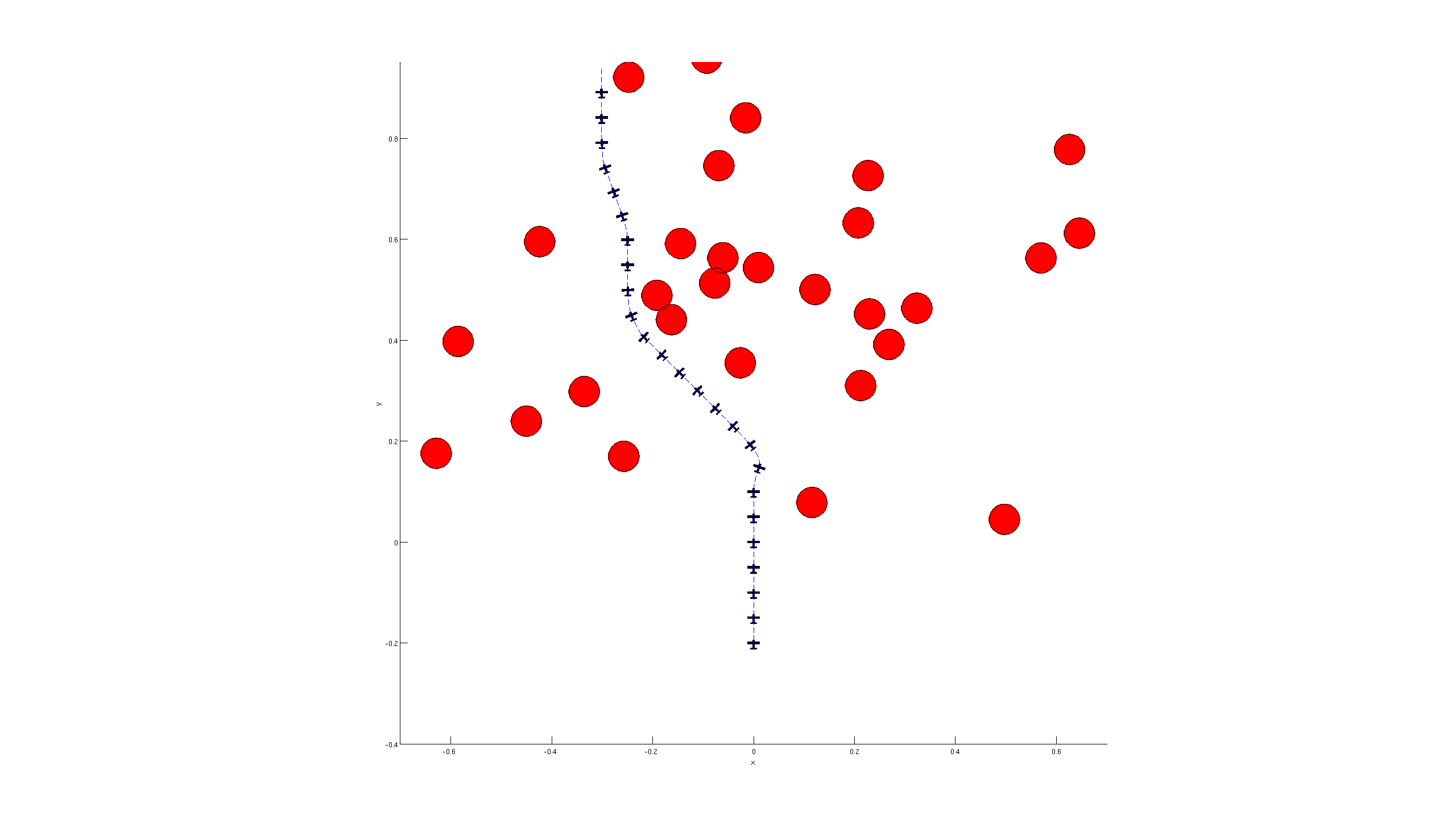}
\label{fig:env2}}
\hfill
\subfigure[Environment 3]{%
\includegraphics[trim = 180mm 30mm 160mm 20mm, clip, width=.25\textwidth]{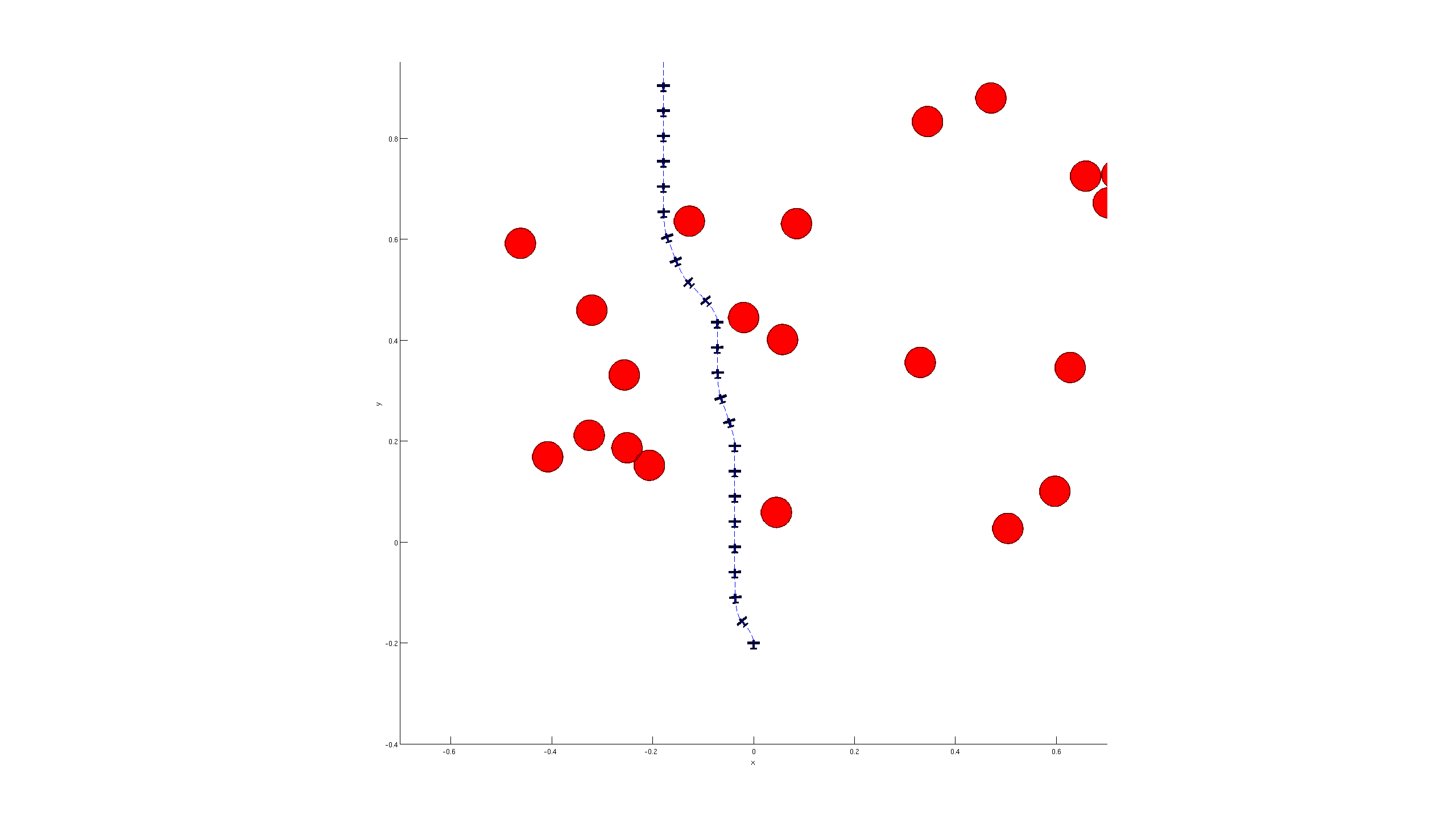}
\label{fig:env3}}
\hfill
\subfigure[Environment 4]{%
\includegraphics[trim = 180mm 30mm 160mm 20mm, clip, width=.25\textwidth]{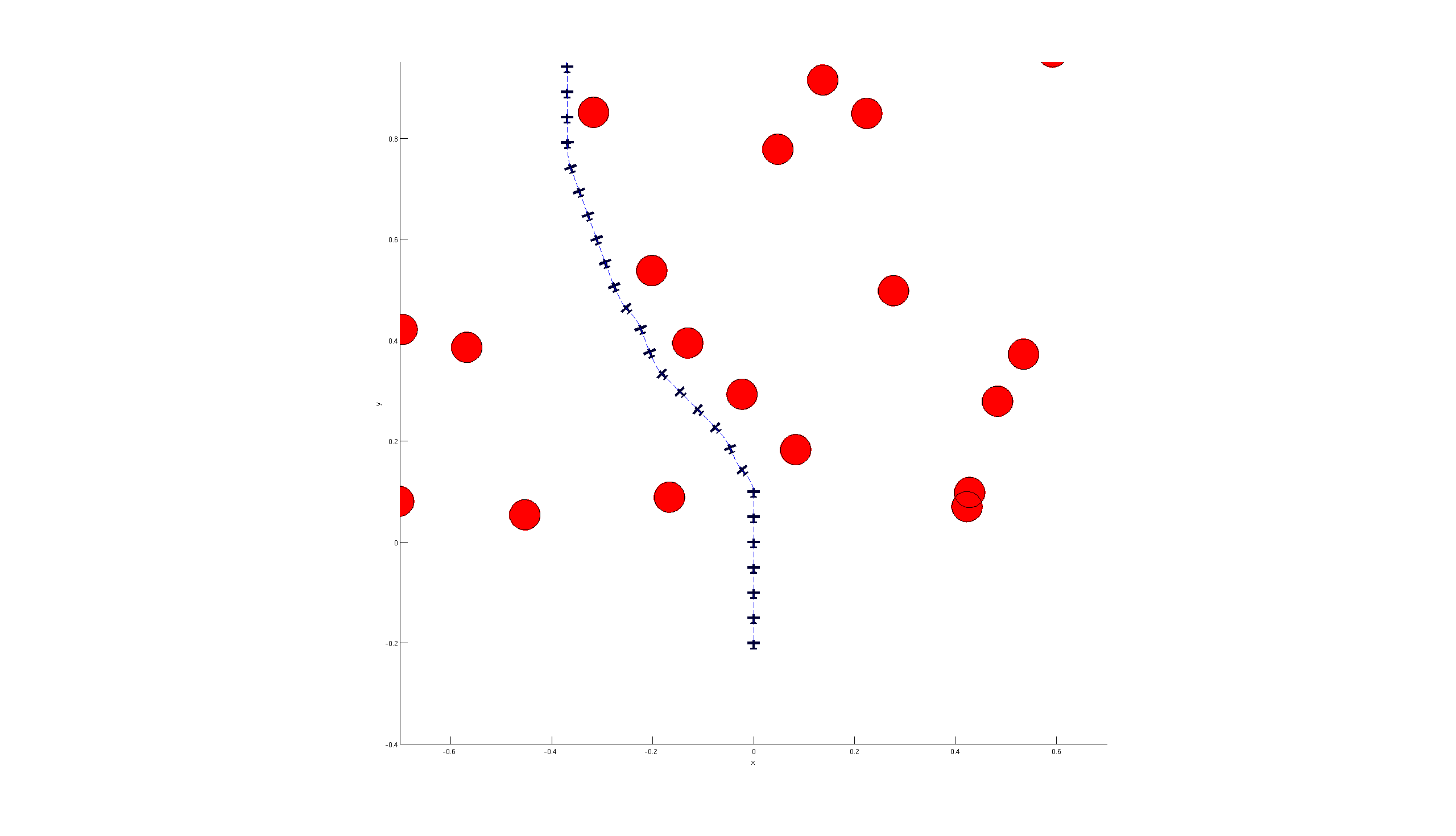}
\label{fig:env4}}
\hfill
\subfigure[Environment 5]{%
\includegraphics[trim = 180mm 30mm 160mm 20mm, clip, width=.25\textwidth]{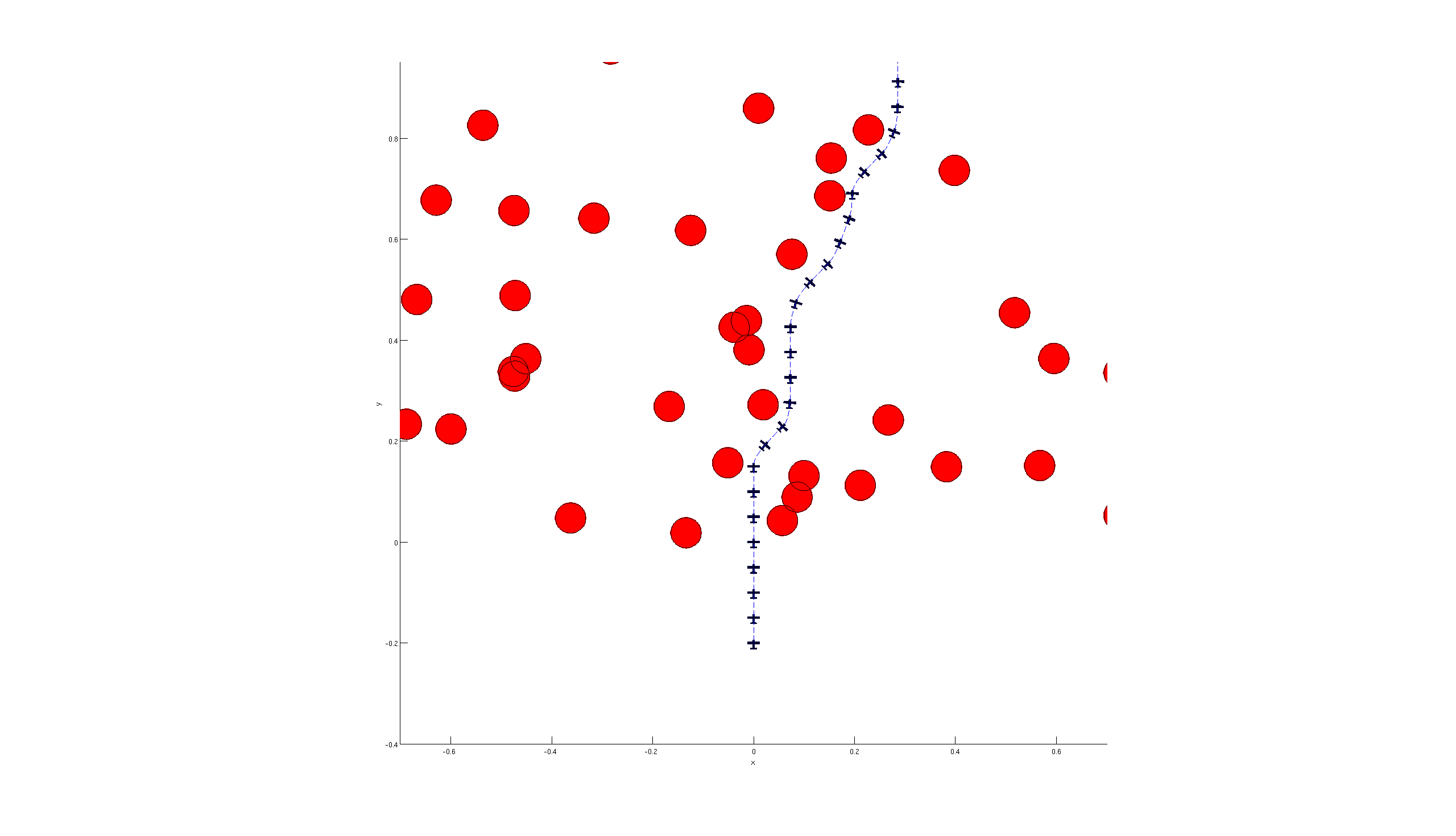}
\label{fig:env5}}
\hfill
\subfigure[Environment 6]{%
\includegraphics[trim = 180mm 30mm 160mm 20mm, clip, width=.25\textwidth]{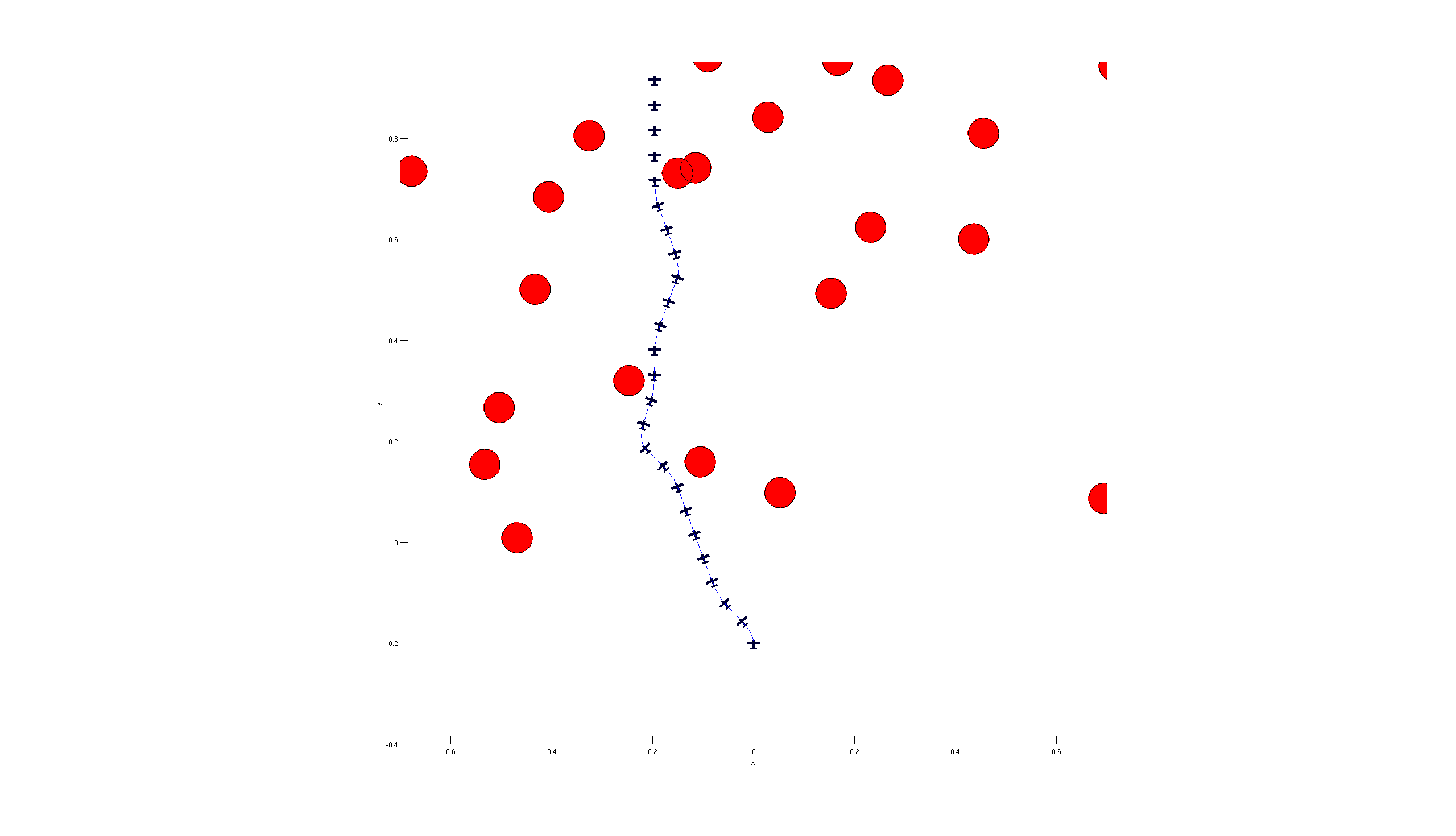}
\label{fig:env6}}
% \hfill
\caption{UAV successfully navigating through different obstacle environments using the planning algorithm described in Section \ref{sec:barriers}. A video of the navigation can be found at \href{http://youtu.be/J3a6v0tlsD4}{http://youtu.be/J3a6v0tlsD4}. This video also shows the barrier certificates (not shown here) as they get updated in real time.}
\label{fig:uav_envs}
\vspace{-10pt}
\end{figure}

Our goal is to make the UAV navigate through cluttered environments that are \emph{unknown pre-runtime}  while avoiding collisions with the obstacles in the environment \emph{despite} the effects of the cross-wind on the vehicle dynamics. In order to achieve this, we pre-compute five control primitives that the UAV can choose from at runtime. These controllers take the form:
\begin{equation}
u_i({\bf{x}}) = -K (\psi - \psi_{des,i}), \ i = 1,\dots,5.
\end{equation} 
These control primitives cause the UAV to control its yaw angle to a particular angle ($\psi_{des,i}$). We choose $K = 50$ and $\psi_{des,1} = 0$ rad, $\psi_{des,2} = -20\pi / 180$ rad, $\psi_{des,3} = 20\pi / 180$ rad, $\psi_{des,4} = -45\pi / 180$ rad, $\psi_{des,5} = 45\pi / 180$ rad.

The UAV's task is to choose from these control primitives in order to navigate its way through the environment. After executing a particular chosen primitive for a short interval of time ($1/20$ seconds in our case), the UAV replans by choosing a control primitive again. Hence, the key decision that our planning algorithm needs to make is the choice of control primitive given a particular configuration of obstacles in its environment. We take our inspiration from \cite{barry2012safety}, which uses \emph{barrier certificates} for verifying the safety of a controller given a particular set of obstacles (but does not consider the case where obstacle positions are not known beforehand and decisions must be made in real time). Similarly, other previous SOS programming approaches~\cite{majumdar2013robust} to collision avoidance have involved solving SOS programs offline and then using these precomputed results to do planning in real time. In contrast, in our example here, the optimization problems are solved in real-time. We describe our approach below.

At every control iteration, we identify the closest two obstacles in the environment in front of the UAV. We then evaluate each control primitive $u_i$ and check if executing it from our current state will result in the UAV avoiding collision with the obstacles. The first safe controller found is executed. The safety of a controller can be checked by computing barrier functions using the polynomial optimization approaches described in Section \ref{sec:math.background}. Denoting the current state as ${\bf{x}_0} = (x_0,y_0,\psi_0)$ and the obstacle sets as $X_{obs,1} \subset \mathbb{R}^2$ and $X_{obs,2} \subset \mathbb{R}^2$, we use polynomial optimization to search for a function $V({\bf{x}})$ of degree $4$ that satisfies the following conditions:
\begin{flalign}
& V({\bf{x}}_0) = 0, \\
& V({\bf{x}}) > 1, \ \forall (x,y) \in X_{obs,i}, \  i = 1,2, \\
& \dot{V}({\bf{x}},w) = \frac{\partial V}{\partial {\bf{x}}}f({\bf{x}}, u_i({\bf{x}}),w)  < 0, \ \forall {\bf{x}} \in X, \ \forall w \in [-0.05,0.05].
\end{flalign}
Here, $X$ is a ``large" set that the system is guaranteed to remain within for the duration of time for which the control primitive is executed. In particular, we choose it to be the unit sphere around the current state. The conditions above imply that the state ${\bf{x}}$ is constrained to evolve within the $1$-sublevel set (in fact the $0$-sublevel set) of the function $V({\bf{x}})$ and is thus \emph{guaranteed} to not collide with the obstacles despite the effects of the cross-wind. 

Hence, at each control iteration we need to solve a maximum of $5$ optimization problems, all of which are independent and can be parallelized. In our example, we use SDSOS programming to compute barrier functions and observe running times of approximately $0.02-0.03$ seconds for feasible problems and $0.08-0.09$ seconds for infeasible problems (i.e., problems where no barrier function can be found) using the Gurobi SOCP solver~\cite{gurobi} (a more thorough running time analysis is presented later). Hence, a real-time implementation of this approach on a hardware platform is plausible. Such a hardware implementation can benefit from already-existing SOCP solvers that are specifically designed to run on embedded systems~\cite{cvxgen},~\cite{boyd_real_time_convex_signal_processing}. In particular, \cite{ecos} presents an approach for generating stand-alone C code for an SOCP solver that can run very efficiently and with low memory footprint. The use of such real-time SOCP solvers has already been considered for tasks such as landing of spacecraft (e.g., for NASA's Mars exploration project)~\cite{Mars_SOCP}.

%These advances in real-time SOCP solvers have already led to hardware implementations, e.g., for attitude control for a spacecraft \cite{Mars_SOCP}.

A particular example of a barrier function computed for the controller $u_1$ is shown in Figure \ref{fig:barrier_plot}. The obstacles are shown in red and the initial state of the UAV is also plotted. The 1-level set of the computed barrier is plotted in green and certifies that the initial state is \emph{guaranteed} to remain safe when the controller is executed.

Figure \ref{fig:uav_envs} demonstrates the performance of the algorithm described above with SDSOS programming used to compute barrier functions on a number of environments. Each subfigure shows a randomly chosen environment (with obstacle positions chosen from the uniform distribution) with circular obstacles that the UAV has to navigate. The trajectory traversed by the UAV following the described planning algorithm is indicated in these plots and remains collision free in each case. Note that the original (non-Taylor expanded) dynamics are used for the simulations.

We end the discussion of this example by comparing running times and performance of the SDSOS and SOS approaches to this problem. In order to do this, we fix the initial state of the vehicle to be $(0,0,\psi_0)$ for varying values of $\psi_0$. For each $\psi_0$, we randomly sample $100$ different environments containing two obstacles each. The obstacles are disks of radius $0.03$ m with centers $(x_c,y_c)$ uniformly sampled in the range $x_c \in [-0.2,0.2]$ m, $y_c \in [0,0.2]$ m. For each environment, we attempt to find a valid barrier certificate for the first controller in our library (i.e., the one that servos the vehicle to $\psi_{des,1} = 0$). The results are summarized in Table \ref{tab:comparison_barrier} which presents the number of environments (out of 100) for which a barrier certificate was successfully found using SDSOS and SOS programming. As the table illustrates, the number of times SDSOS programming \emph{fails} to find a barrier certificate when SOS programming \emph{succeeds} is quite small.

\begin{table}
\small
\begin{center}
  \begin{tabular}{ | c | c | c | c | c | c |  }
    \hline 
    $\psi_0$ & $0^{\circ}$ & $10^{\circ}$ & $20^{\circ}$ & $30^{\circ}$ & $40^{\circ}$ \\ \hline
    SDSOS & 66 $\%$  & 59 $\%$ & 70 $\%$  & 68 $\%$ & 56 $\%$  \\ \hline
    SOS & 68 $\%$ & 62 $\%$  & 70 $\%$  & 76 $\%$  & 65 $\%$  \\ \hline
  \end{tabular}
    \caption{Comparison of percentage of times a valid barrier certificate was found using SDSOS and SOS programming for randomly sampled obstacle environments and initial yaw angles. (Only the ratio between the two is meaningful here.) \label{tab:comparison_barrier}}
  \end{center}
  \vspace{-10pt}
  \end{table}

We also compare running times of the two approaches in Figure \ref{fig:comparison_barrier_runtimes}. We use the Gurobi SOCP solver  \cite{gurobi}  for the SDSOS problems and SeDuMi \cite{sedumi} as the SDP solver for SOS problems. As the histograms of running times illustrate, the SDSOS approach is significantly faster than the SOS approach. We note that while the MOSEK SOCP/SDP solvers are typically faster, we were unable to make these work on this problem due to numerical issues.

\begin{figure}[H]
\centering
\includegraphics[trim = 0mm 0mm 0mm 0mm, clip, width=.4\textwidth]{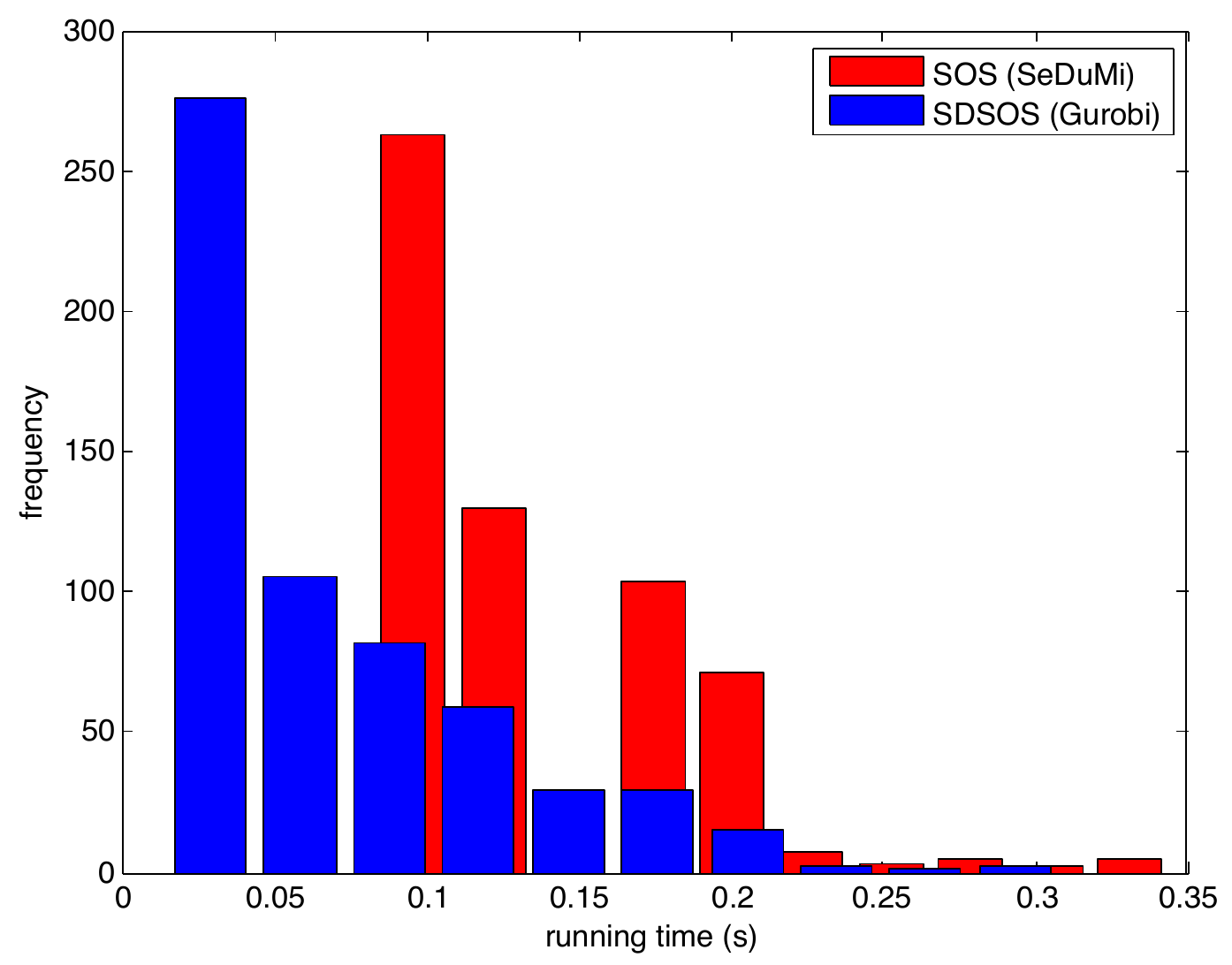}
\caption{Histograms of running times for SOS and SDSOS approaches on the collision avoidance problem.}
\label{fig:comparison_barrier_runtimes}
\end{figure}

\section{Nonlinear Control Design for a Quadrotor Model}
\label{sec:quadrotor}

Quadrotors (see Figure \ref{fig:quad_pic}) have recently been recognized as a popular platform for academic research in systems theory due to their agile maneuvering capabilities and inexpensive cost \cite{Mellinger11a,Hoffmann07}. They have also been considered for the task of load transportation, not only in laboratory settings \cite{Mellinger10c}\footnote{A video corresponding to the paper is available at \href{https://www.youtube.com/watch?v=YBsJwapanWI}{https://www.youtube.com/watch?v=YBsJwapanWI}}, but also by the aerospace companies Bell and Boeing and the online retail company Amazon\footnote{\href{https://www.youtube.com/watch?v=Le46ERPMlWU}{https://www.youtube.com/watch?v=Le46ERPMlWU}}. In this section, we consider the problem of designing a nonlinear stabilizing feedback controller for the quadrotor's hovering configuration, which is relevant to almost all of its applications. In addition to a stabilizing controller, we also obtain a \emph{formal certificate} of stability of the resulting system. This certificate takes the form of an inner approximation of the region of attraction (ROA), i.e., the set of initial conditions the controller is guaranteed to stabilize to the goal position.

\begin{figure}[H]
\centering
\includegraphics[trim = 0mm 0mm 0mm 0mm, clip, width=.3\textwidth]{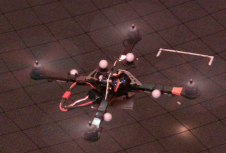}
\caption{We design a hovering controller for the quadrotor model described in \cite{Mellinger10}. (Image from \cite{Mellinger10}.)}
\label{fig:quad_pic}
\end{figure}

We use the dynamics model described in \cite{Mellinger10} for our numerical experiments. The model includes 16 states:
$$x\mathrel{\mathop:}=[x_1,y,z,\phi,\theta,\psi,\dot{x},\dot{y},\dot{z},p,q,r,\omega_1,\omega_2,\omega_3,\omega_4],$$
where $x_1,y,z$ are the coordinates of the center of mass of the system, $\phi,\theta,\psi$ are the Euler angles describing its orientation, $p,q,r$ are angular velocities of the quadrotor expressed in the body frame, and $\omega_i, i = 1,...,4$, are the angular speed of the rotors. The rotor angular speeds cannot be controlled directly and have nontrivial dynamics. The control inputs of the system are thus the \emph{desired} speed of the rotors (the rotors take some time to catch up to the desired speed).

In the end our system takes the form $\dot{x}=f(x)+g(x)u(x)$ with $f$ and $g$ given and the control $u$ as a decision function. We use the method presented in our earlier work \cite{Majumdar13} in collaboration with Russ Tedrake to design a hovering controller $u$ for the system. The fixed point corresponding to the hovering configuration has the first twelve states of the system equaling $0$ but with non-zero rotor speeds $\omega_i$ counteracting the force of gravity. The dynamics of the system are Taylor expanded to degree $3$ in order to obtain polynomial dynamics. We search for a degree $2$ Lyapunov function $V(x)$ and a degree $3$ feedback controller $u(x)$ in order to maximize the size of the region of attraction (ROA) of the resulting closed-loop system (i.e., the differential equation with $u(x)$ plugged in). We use SDSOS programming since the state space is too large for SOS programming to handle, causing our computer to run out of memory. The resulting optimization problem is:
\begin{flalign} \label{eq:control_design}
  \mathop{\textrm{max}}_{\rho,L(x),V(x),u(x)} \quad & \rho \\ 
  \textrm{s.t.} \quad &   V(x)  \in SDSOS_{16,2} \nonumber \\ 
   \quad  -&\dot{V}(x) + L(x)(V(x) - \rho) \in SDSOS_{16,6} \nonumber \\
   \quad & L(x)  \in SDSOS_{16,4}  \nonumber \\ 
   \quad & \sum_j V(e_j) = 1.  \nonumber
   \end{flalign} 
Here, $L(x)$ is a nonnegative multiplier term and $e_j$ is the $j$-th standard basis vector for the state space $\mathbb{R}^{16}$. From our discussion in Section~\ref{sec:Lyap.background}, it is easy to see that the above conditions are sufficient for establishing $B_\rho = \{x \in \mathbb{R}^{16} \ | \ V(x) \leq \rho\}$ as an inner estimate of the region of attraction for the system. When $x \in B_\rho$,  the second constraint implies that $\dot{V}(x) < 0$ (since $L(x)$ is constrained to be nonnegative).  The last constraint normalizes $V(x)$ so that maximizing the level set value $\rho$ leads to enlarging the volume of the ROA. %Note that this normalization does not introduce conservativeness since if a valid Lyapunov function exists, one can always scale it to satisfy this normalization constraint.

The optimization problem \eqref{eq:control_design} is not convex in general since it involves conditions that are bilinear in the decision variables. However, problems of this nature are common in the SOS programming literature (see e.g. \cite{Jarvis-Wloszek03}) and are typically solved by iteratively optimizing groups of decision variables.  Each step in the iteration is then a SDSOS program. This iterative procedure is described in more detail in \cite{Majumdar13} and can be initialized with the Lyapunov function from a Linear Quadratic Regulator (LQR) controller~\cite{Anderson90}. The iterations are terminated when the objective changes by less than $1$ percent.
% and a small enough value of $\rho$. 

\begin{figure}[H]
\centering
\subfigure[x -- y subspace]{%
\includegraphics[trim = 20mm 70mm 25mm 70mm, clip, width=.25\textwidth]{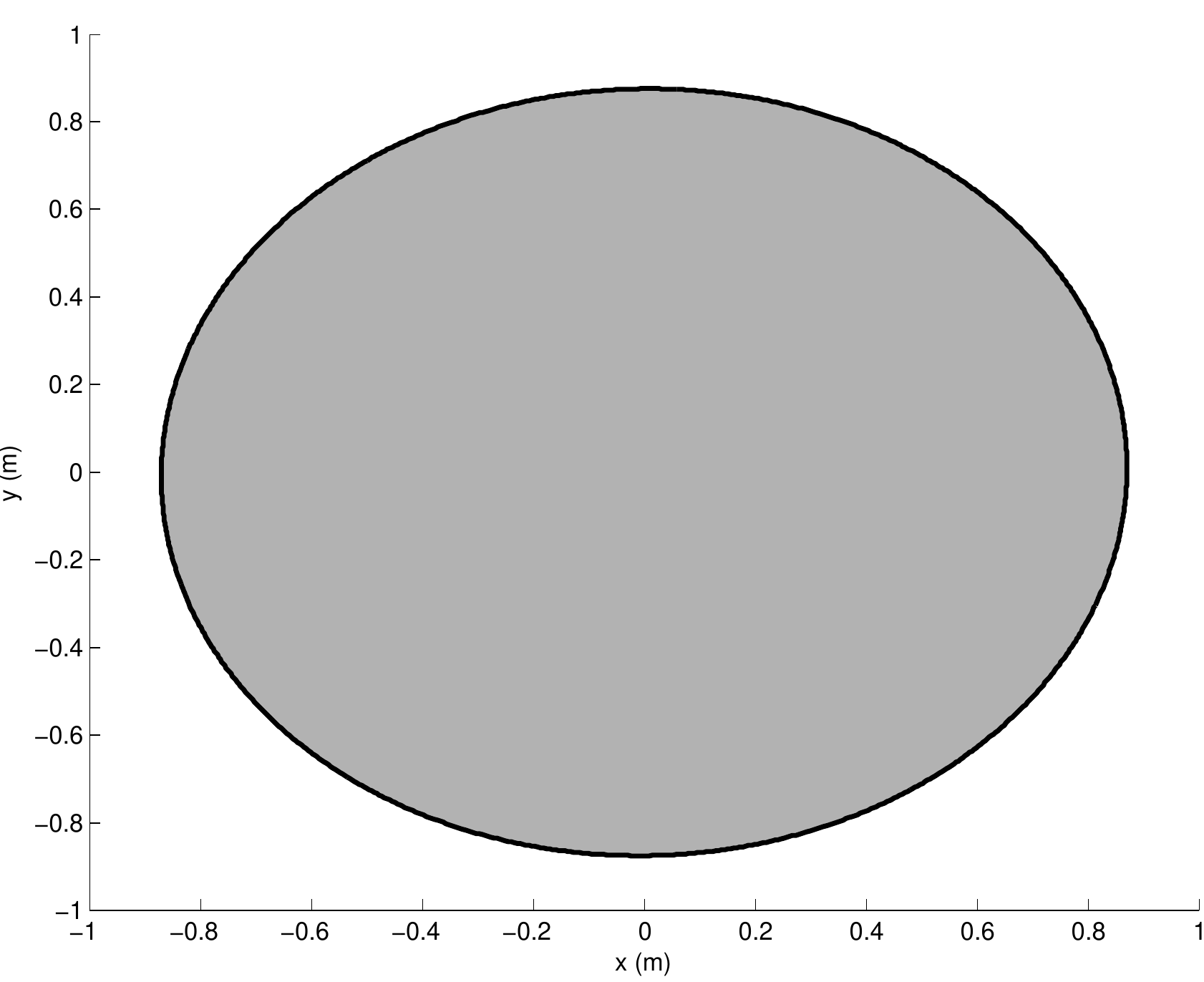}
\label{fig:x_vs_y}}
\hfill
\subfigure[x -- z subspace]{%
\includegraphics[trim = 20mm 70mm 25mm 70mm, clip, width=.25\textwidth]{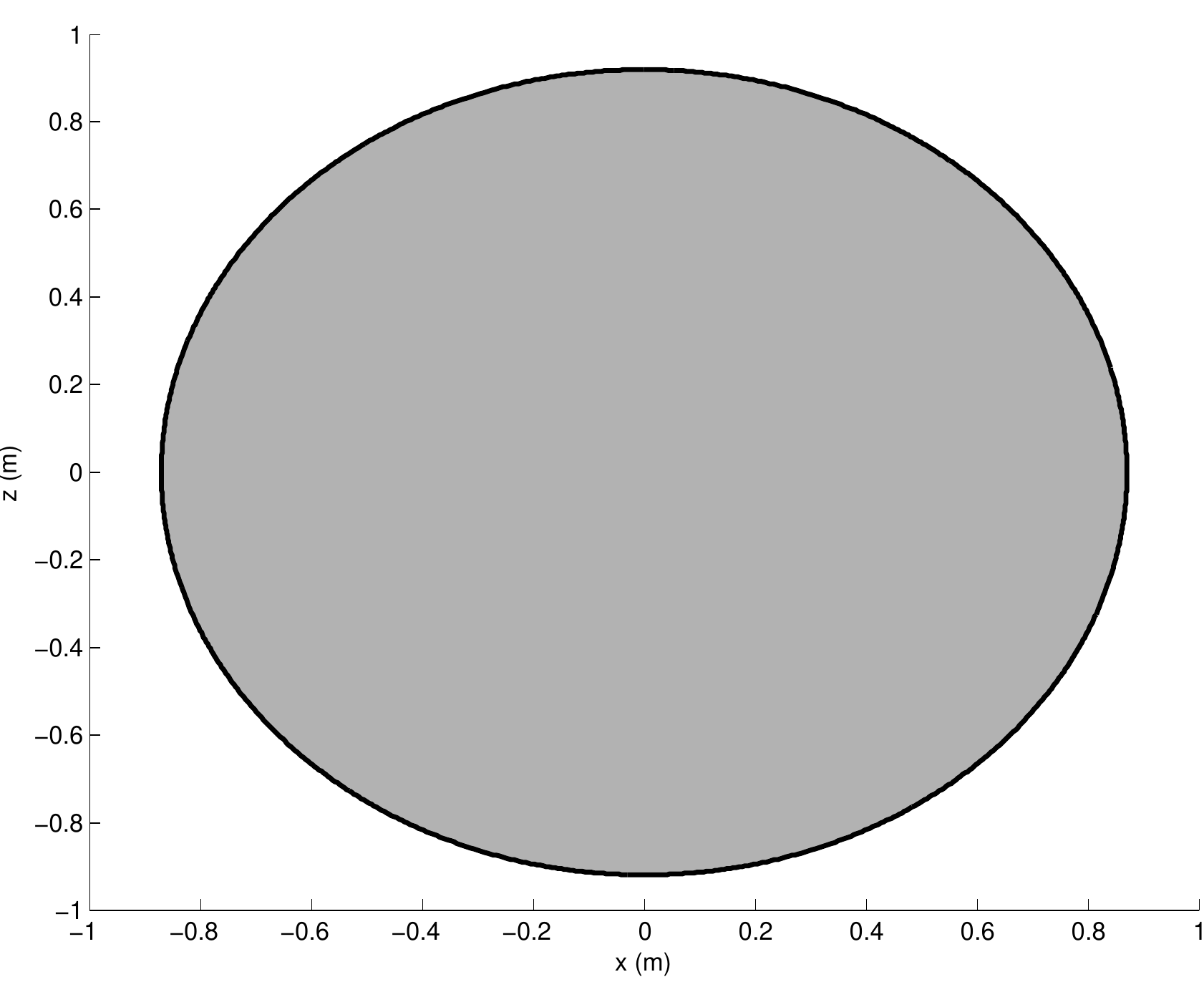}
\label{fig:x_vs_z}}
\hfill
%\subfigure[x -- roll subspace]{%
%\includegraphics[trim = 20mm 70mm 25mm 70mm, clip, width=.3\textwidth]{Figures/quad_x_vs_roll.pdf}
%\label{fig:x_vs_roll}}
%\hfill
\subfigure[x -- pitch subspace]{%
\includegraphics[trim = 20mm 70mm 25mm 70mm, clip, width=.25\textwidth]{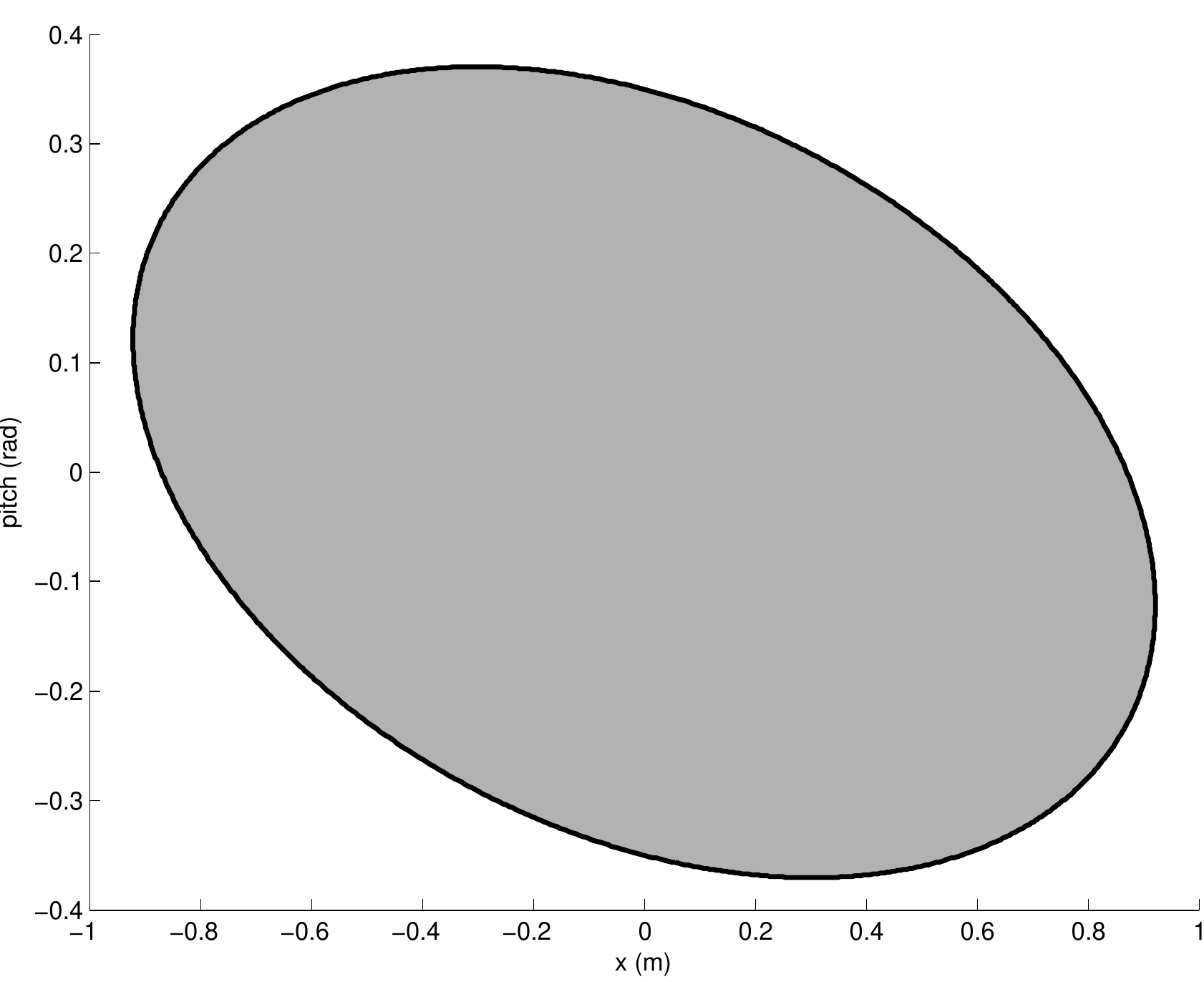}
\label{fig:x_vs_pitch}}
\caption{Slices in different subspaces of the hovering ROA of the quadrotor system.}
\label{fig:quad_roa}
\end{figure}

An important observation is that unlike the sets $POS_{n,d}$ and $SOS_{n,d}$, the sets $DSOS_{n,d}$ and $SDSOS_{n,d}$ are not invariant to coordinate transformations, i.e., a polynomial $p(Ax)$ is not necessarily dsos (resp. sdsos) even if $p(x)$ is dsos (resp. sdsos). Thus, performing coordinate transformations on the problem data (e.g., on the state variables of a dynamical system) can sometimes have an important effect. We describe a particular coordinate transformation that is intuitive and straightforward to implement. It can be used for problems involving the search for Lyapunov functions, and can potentially be extended to other problems as well. In particular, given a Lyapunov function $V(x)$ we find an invertible affine transformation that simultaneously diagonalizes the Hessians of $V(x)$ and $-\dot{V}(x)$ evaluated at the origin (this is always possible for two positive definite matrices). The intuition behind the coordinate change is that the functions $V(x)$ and $-\dot{V}(x)$ locally resemble functions of the form $x^TDx$ (with $D$ diagonal), which are dsos polynomials that are ``far away'' from the boundary of the DSOS (and hence SDSOS) cone. We solve the optimization problem \eqref{eq:control_design} after performing this coordinate transformation. The transformation is then inverted to obtain ROAs in the original coordinate frame. % We use the Lyapunov function corresponding to a LQR controller to perform this coordinate transformation.

\begin{figure}[H]
\centering
\includegraphics[trim = 220mm 5mm 230mm 90mm, clip, width=.3\textwidth]{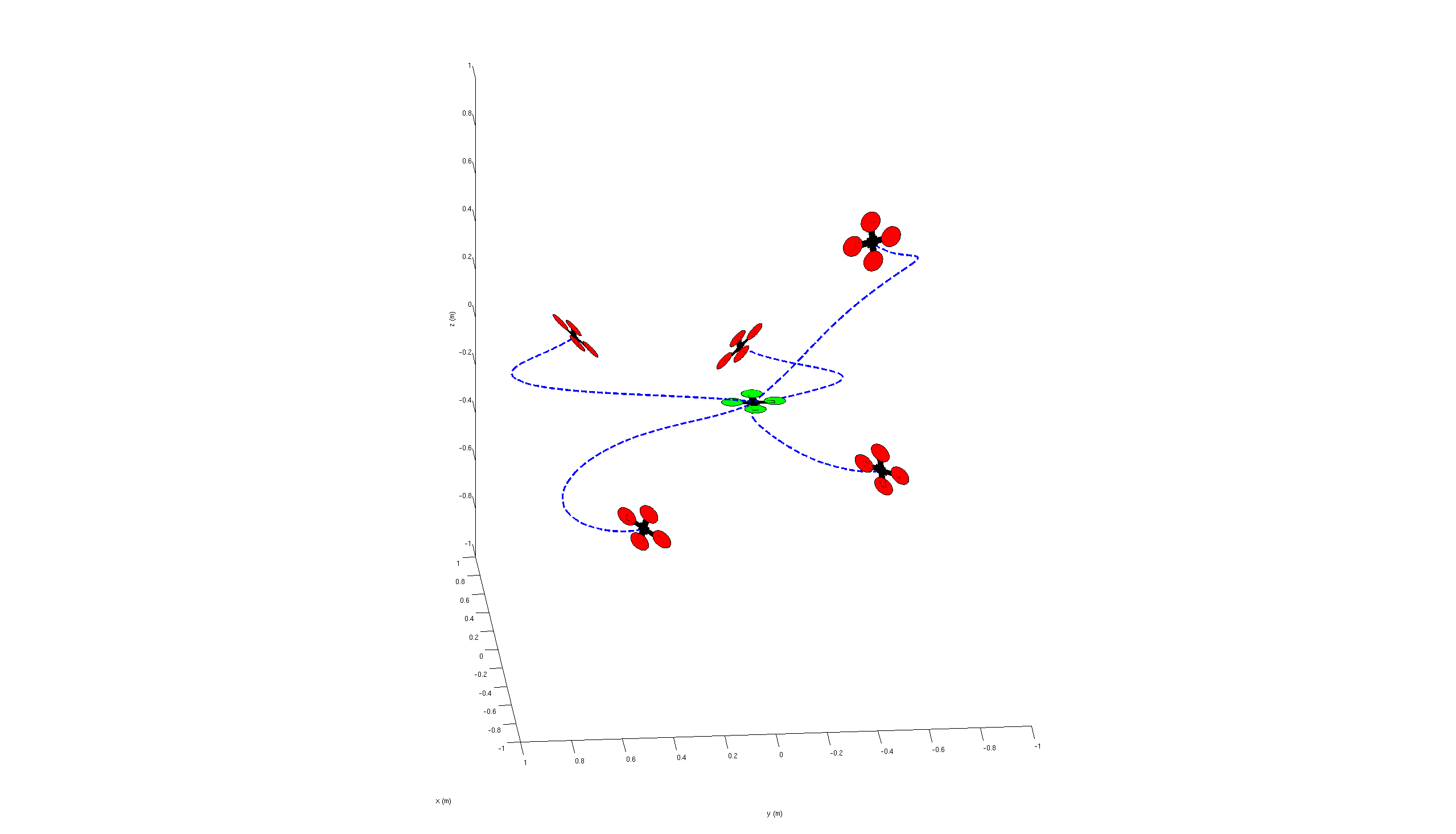}
\caption{A sampling of five initial conditions that are stabilized by our controller. The goal position is shown in green, the stabilized initial conditions in red, and the intermediate trajectories in blue.}
\label{fig:quad_snaps}
\vspace{-10pt}
\end{figure}

Each iteration of the algorithm employed for solving the optimization problem \eqref{eq:control_design} takes approximately 15 minutes, with convergence occurring between 15 and 20 iterations. Figure \ref{fig:quad_roa} shows slices of the computed ROA in multiple subspaces of the state space. As the plot illustrates, we are able to verify stability of the closed loop system for a large set of initial conditions. A qualitative demonstration of the performance of the controller is given in Figure \ref{fig:quad_snaps}. The system is started off from five different initial conditions (shown in red) and our nonlinear hovering controller is applied. The resulting trajectory is shown in blue. In each case the quadrotor is able to stabilize itself to the goal configuration (green).

\section{Conclusions}\label{sec:conclusions}
In this paper, we demonstrated three applications of optimization problems over the set of nonnegative polynomials that may be of interest in operations research and transportation engineering. We hope to have conveyed the message that the problem of certifying polynomial inequalities appears in more diverse areas than one might think. There are powerful tools for approaching this problem based on the sum of squares relaxation and semidefinite programming. We believe that our recently introduced techniques of DSOS and SDSOS optimization, which are LP and SOCP-based alternatives to sum of squares programming, can pave the way to new applications of algebraic techniques in optimization---in particular, applications that are large-scale or real-time.

\small{
\section{Acknowledgements}
The authors would like to thank Pablo Parrilo, Russ Tedrake, and the MIT Robot Locomotion Group for many helpful discussions that have contributed greatly to this paper. The authors would also like to acknowledge the use of the software package Drake (\href{https://github.com/RobotLocomotion/drake/wiki}{https://github.com/RobotLocomotion/drake/wiki}) developed by the Robot Locomotion Group for formulating the dynamics in the quadrotor example, along with the SPOTless software developed by Mark Tobenkin, Frank Permenter and Alexandre Megretski for processing the SOS programs in our examples. Finally, we are very grateful for receiving constructive criticism from a referee that led to improvements in this paper.
}

%applications that start to move towards 5large-scale or real-time.

%\newpage
\begin{multicols}{2}
%\small{
\bibliographystyle{abbrv}
% \bibliography{elib}
\bibliography{pablo_amirali,elib}
% \bibliography{pablo_amirali}
%}
\end{multicols}

%%%%%%%%%%%%%%%%%%%%%%%%%%%%%%%%%%%%%%%%%%%%%%%%%%%%%%%%%%%%%%%%%%%%%%%%%%%%%%%

%%%%%%%%%%%%%%%%%%%%%%%%%%%%%%%%%%%%%%%%%%%%%%%%%%%%%%%%%%%%%%%%%%%%%%%%%%%%%%%
%\bibliographystyle{unsrt}

\end{document}